\documentclass[12pt,reqno]{amsart}

\usepackage[margin=1in]{geometry}

\title[Global well-posedness for active scalars equation]{Global well-posedness of slightly supercritical active scalar equations}

\date{\today}

\author{Michael Dabkowski}
\address{Department of Mathematics, University of Toronto}
\email{michael.dabkowski@utoronto.ca}

\author{Alexander Kiselev}
\address{Department of Mathematics, University of Wisconsin}
\email{kiselev@math.wisc.edu}

\author{Luis Silvestre}
\address{Department of Mathematics, The University of Chicago}
\email{\tt luis@math.uchicago.edu}

\author{Vlad Vicol}
\address{Department of Mathematics, The University of Chicago}
\email{\tt vicol@math.uchicago.edu}

\usepackage{amsfonts,amsmath,latexsym,amssymb,verbatim,amsbsy,amsthm,mathrsfs}

\usepackage{times}

\usepackage{hyperref}
\usepackage{color}

\theoremstyle{plain}
\newtheorem{theorem}{Theorem}[section]
\newtheorem{definition}[theorem]{Definition}
\newtheorem{lemma}[theorem]{Lemma}
\newtheorem{proposition}[theorem]{Proposition}

\theoremstyle{definition}
\newtheorem{remark}[theorem]{Remark}

\def\tilde{\widetilde}
\numberwithin{equation}{section}

\renewcommand\hat{\widehat}

\newcommand\blu[1]{{#1}}

\def\RR{{\mathbb R}}
\def\TT{{\mathbb T}}
\def\DD{\mathcal D}

\def\eps{\varepsilon}
\def\LL{\Lambda}
\def\omb{\omega_B}
\def\EL{\mathcal L}
\def\zz{\zeta}

\def\just{slightly\ }

\begin{document}


\begin{abstract}
The paper is devoted to the study of \just supercritical active scalars with nonlocal diffusion. We prove global regularity for the surface quasi-geostrophic (SQG) and Burgers equations, when the diffusion term is supercritical by a symbol with roughly logarithmic behavior at infinity. We show that the result is sharp for the Burgers equation. We also prove global regularity for a \just supercritical two-dimensional Euler equation. Our main tool is a nonlocal maximum principle which controls a certain modulus of continuity of the solutions.
\end{abstract}

\subjclass[2000]{35Q35,76U05}
\keywords{Surface quasi-geostrophic equation, Burgers equation, supercritical, global regularity, finite time blow-up, nonlocal maximum principle, nonlocal dissipation.}

\maketitle

\section{Introduction} \label{sec:intro}

Active scalars play an important role in fluid mechanics. An active scalar equation is given by
\begin{equation}\label{as1}
\partial_t \theta + (u \cdot \nabla)\theta + \EL \theta = 0, \,\,\,\theta(x,0)=\theta_0(x),
\end{equation}
where $\EL$ is typically some dissipative operator, such as fractional dissipation, and $u$ is the flow velocity that is determined by $\theta.$ A common setting is either on $\RR^d$ or $\TT^d.$
Active scalar \blu{equations} are nonlinear, and most active scalars of interest are nonlocal. This makes the analysis of these equations challenging. The best known active scalar
equations are \blu{the} 2D Euler equation (for which $u=\nabla^\perp (-\Delta)^{-1}\theta$), \blu{the} surface quasi-geostrophic (SQG) equation ($d=2,$ $u=\nabla^\perp (-\Delta)^{-1}\theta$), and one-dimensional
Burgers equation ($u=\theta$). The 2D Euler and Burgers equations are classical \blu{in fluid mechanics}, while the SQG equation was first considered in the mathematical literature by Constantin, Majda and Tabak \cite{CMT}, and since then
has attracted significant attention, in part due to certain similarities with three dimensional Euler and Navier-Stokes equations.

Observe that for SQG and Burgers \blu{the drift velocity} $u$ and \blu{the advected scalar} $\theta$ are of the same order of regularity, while for 2D Euler $u$ is more regular by a derivative. The 2D Euler equation has global regular
solutions, and can be thought of as critical case. For the Burgers and SQG equations, fractional dissipation $\EL = \LL^{\alpha},$ where $\LL=(-\Delta)^{1/2}$ is the Zygmund operator, have been often considered. Both of these equations possess $L^\infty$ maximum principle \cite{R,CC}, and this makes $\alpha =1$ critical \blu{with respect to the natural scaling of the equations}. It has been known for a while that in the subcritical case $\alpha >1$ global regular solutions exist for sufficiently smooth initial data (see \cite{R} for SQG equation; \blu{the} analysis for Burgers is very similar). The critical case $\alpha=1$ has been especially interesting for the SQG equation since it
is well motivated physically, with $\Lambda \theta$ term modeling so called Ekman boundary layer pumping effect (see e.g. \cite{Pedl}). The global regularity in the critical case has been settled
independently by Kiselev-Nazarov-Volberg \cite{KNV} and Caffarelli-Vasseur \cite{CaV}. A third proof of the same
result was provided by Kiselev and Nazarov in \cite{KN2}, and a fourth \blu{recently by Constantin and Vicol} in \cite{CoVi}. All these proofs are quite different. The method of \cite{CaV} is inspired by DeGiorgi iterative estimates, while the
duality approach of \cite{KN2} uses \blu{an} appropriate set of test functions and estimates on their evolution. The proof in \cite{CoVi} takes advantage of \blu{a} new nonlinear maximum principle, \blu{which gives a nonlinear bound on a linear operator}. The method of \cite{KNV}, on the other hand, is based on a technique which may be called a nonlocal maximum principle. The idea is to prove that the evolution \eqref{as1}
preserves a certain modulus of continuity $\omega$ of the solution. In the critical SQG case, the control is strong enough to give a uniform bound on $\|\nabla
\theta\|_{L^\infty}$, which is sufficient for
global regularity.

In the supercritical case, until recently the only results available (for large initial data) have been on conditional regularity and finite time regularization of solutions. It was shown by Constantin and Wu~\cite{CW} that if the solution is $C^{\delta}$ with $\delta > 1-\alpha$, then it is smooth (see also Silvestre~\cite{S2} for drift velocity that is not divergence free). Dong and Pavlovic~\cite{DP} later improved this result to $\delta = 1-\alpha$. Finite time regularization has been proved by Silvestre \cite{S} for $\alpha$
sufficiently close to $1$, and for the whole dissipation range $0<\alpha<1$ by Dabkowski \cite{D} (with an alternative proof of the latter result given in \cite{K2}). The issue of global
regularity in the case $\alpha\in(0,1)$ remains an outstanding open problem. Recently, a small advance into the supercritical regime was made in \cite{DKV}, where the SQG equation with velocity
given by $u = \nabla^\perp \Lambda^{-1}m(\Lambda)\theta$ was considered. Here $m$ is a Fourier multiplier which \blu{may} grow at infinity at any rate slower
than double logarithm. The method of \cite{DKV} was based on the technique of \cite{KNV}. The main issue is that even with very slow growth of $m,$ the equation loses scaling, which plays an
important role in every proof of regularity for the critical case. The work \cite{DKV} was partly inspired by the \just supercritical Navier-Stokes regularity result of Tao \cite{Tao}, and
partly by the recent work of \cite{CCW,CCW2} on generalized Euler and SQG models.

In this paper, we \blu{analyze} a \just supercritical SQG and Burgers equations. As opposed to \cite{DKV}, we keep the velocity definition the same as for classical SQG and Burgers, and \blu{instead treat supercritical diffusion}. We also consider nonlocal diffusion terms more general than fractional Laplacian, including cases where the $L^\infty$ maximum principle does not hold. We show, roughly,
that symbols supercritical by a logarithm or less lead to global regular solutions for both equations. Our main technique is the control of appropriate family of moduli of continuity of the
solutions. \blu{For the Burgers equation, when the conditions we impose on the diffusion in order to obtain global regularity are not satisfied, then we prove that some smooth initial data leads to finite time blow up. In this respect, our well posedness result is sharp.} For the SQG, the global regularity proof is more sophisticated than for the Burgers equation. The upgrade from the double logarithm supercriticality of \cite{DKV} to
logarithmic one is made possible by exploiting the structure of nonlinearity, in particular the $\nabla^\perp$ in $u = \nabla^\perp \Lambda^{-1}\theta.$ This idea is based on \cite{K2}, where
this structure was exploited to prove finite time regularization for power supercritical SQG. We note that Xue and Zheng \cite{XZ} observed a similar improvement from $\log\log$ to $\log$ in the context of supercritical velocity.

We also consider \just supercritical 2D Euler equation, and generalize the results of \cite{CCW} on global regularity of solutions.

Below, we state main results proved in the paper. In Section~\ref{sec:prelim:MOC}, we provide some basic background results on the nonlocal maximum principles.
Section~\ref{sec:SQG:global} is devoted to proof of global regularity for \just supercritical SQG with nonlocal diffusions. The Burgers case is handled in Section~\ref{sec:Burg:global}.
In Section~\ref{app:Fourier}, we consider the case of dissipation given by Fourier multipliers. Some natural multipliers can lead to non-positive convolution kernels for the corresponding
nonlocal diffusion, and we generalize our technique to this case. Section~\ref{sec:Euler} is devoted to \just supercritical 2D Euler equation.

To state our main results, we need to introduce some notation.

\subsection{Assumptions on $m$} \label{sec:m:assume}
Let $m \colon (0,\infty) \to [0,\infty)$ be a non-increasing smooth function which is singular at the origin, i.e. $\lim_{r \to 0} m(r) = \infty$, and satisfies the following conditions:
\begin{itemize}
\item[(i)] there exists a sufficiently large positive constant $C_0>0$ such that
\begin{align}
r m(r) \leq C_0 \mbox{ for all } r\in(0,r_0) \label{eq:m:critical}
\end{align}
for some $r_0>0$. The above condition is natural in the present context, since otherwise the dissipative operator defined below is {subcritical}, which is not the purpose of this paper.
\item[(ii)]  there exits some $\alpha > 0$ such that
\begin{align}
r^{\alpha} m(r) \mbox{ is non-increasing.} \label{eq:m:decay}
\end{align}
The above assumption is slightly stronger than just having $m(r)$ be non-increasing. 
\end{itemize}
Throughout this paper we also denote  by $m$ the radially symmetric function $m \colon \RR^{d}\setminus \{0\} \to (0,\infty)$ such that $m(y) = m(|y|)$ for each $y \in \RR^d \setminus \{0\}$.
Note that the above conditions allow for $m$ to be identically zero on the complement of a ball.

The examples of functions $m$ which are relevant to this paper are those that are less singular than $r^{-1}$ near $r=0$. These functions yield dissipative nonlocal operators
(cf.~\eqref{eq:L:def} below) that are less smoothing than $\Lambda$, which makes the corresponding SQG and Burgers equation {\em supercritical}. The main examples we have in mind are
\begin{align}
m(r)  = \frac{1}{r^{a}}\,\,\,{\rm and}\,\,\,m(r) = \frac{1}{r (\log(2/r))^{a}} \label{eq:m:log}
\end{align}
with $0< a \leq 1$, $0<r \leq 1,$ coupled with enough regularity and decay for $r>1$. 
The first class corresponds to power supercticality. The second class, supercritical by a logarithm, is relevant for the global well-posedness results we prove. 
It is not hard to verify that the functions in \eqref{eq:m:log} verify \eqref{eq:m:critical}--\eqref{eq:m:decay} on
$(0,1]$, and that they can be suitably extended on $[1,\infty)$.

\subsection{Dissipative nonlocal operators}  \label{sec:nonlocal:operators}
Associated to any such function $m$ we consider the nonlocal operator
\begin{align}
\EL \theta(x) = \int_{\RR^{d}} \big( \theta(x) - \theta(x+y) \big) \frac{m(y)}{|y|^d} dy \label{eq:L:def}.
\end{align}
Above and throughout the rest of the paper the integral is meant in principal value sense, but we omit the $P.V.$ in front of the integral. For example, when $m(r) = r^{-\alpha} C_{d,\alpha}$
for a suitable normalizing constant $C_{d,\alpha}$, then $\EL = \LL^{\alpha}$. 
The nonlocal operators $\EL$ we consider here are
dissipative  because $m$ is singular at the origin: due to \eqref{eq:m:decay}, we have that $m(r) \geq m(1) r^{-\alpha}$ for some $\alpha>0$ when $r \leq 1$, so that $\EL$ is at least as
dissipative as $\LL^{\alpha}$. It is now evident that when $\lim_{r\to 0} r m(r) = 0$, the corresponding nonlocal operator $\EL$ is less smoothing than $\LL$. We emphasize that for $\theta \in
C^\infty(\TT^{d})$ the P.V. integral in \eqref{eq:L:def} converges only if $m$ is sub-quadratic near $r=0$, i.e.
\begin{align*}
 \int_{0}^{1} r m(r) dr < \infty
\end{align*}
holds. In our case, the above condition is satisfied in view of assumption \eqref{eq:m:critical}. Convergence near infinity is not an issue due to the assumption \eqref{eq:m:decay}.

All results in this paper can be generalized to a more general class of dissipative operators.
Namely, for each function $m$ that satisfies \eqref{eq:m:critical}--\eqref{eq:m:decay}, consider the class of smooth radially symmetric kernels $K \colon \RR^d\setminus\{0\} \to (0,\infty)$
which satisfy
\begin{align}
\frac{m(y)}{\blu{C}\, |y|^d} \leq K(y) \leq \frac{\blu{C}\, m(y)}{|y|^d}
\label{eq:K:cond}
\end{align}
for some constant $C > 0$, and all $y \neq 0$. Associated to each such kernel $K$ we may consider the dissipative nonlocal operator
\begin{align}
\EL \theta(x) = \int_{\RR^{d}} \big( \theta(x) - \theta(x+y) \big) K(y) dy \label{eq:L:general}
\end{align}
which generalizes the definition in \eqref{eq:L:def}. As we will see, such generalization will be useful when working with dissipative operators
generated by Fourier multipliers. Moreover, as we will see later in the paper, conditions on $K$ can be relaxed further.

\subsection{Main results}\label{sec:results}

The generalized dissipative SQG equation reads
\begin{align}
&\partial_{t} \theta  + u \cdot \nabla \theta + \EL \theta = 0 \label{eq:SQG:1}\\ &u = \nabla^{\perp} \Lambda^{-1} \theta \label{eq:SQG:2}
\end{align}
where $\EL$ is as defined in \eqref{eq:L:def}, and $m$ is as described above. The main result of this paper
with respect to the dissipative SQG equation is:

\begin{theorem}[\bf Global regularity for \just supercritical SQG]\label{thm:SQG:global}
 Assume that $\theta_0$ is smooth and periodic, and $m$ satisfies an additional assumption  
\begin{align}
\lim_{\eps \to 0+} \int_{\eps}^{1} m(r) dr  = \infty. \label{eq:m:SQG}
\end{align}
Then there exists a unique $C^\infty$ smooth solution $\theta$ of the initial value problem associated to \eqref{eq:SQG:1}--\eqref{eq:SQG:2}.
\end{theorem}

In analogy, one may consider the generalized dissipative Burgers equation
\begin{align}
&\partial_{t} \theta  - \theta \theta_x + \EL \theta = 0 \label{eq:Burg}
\end{align}
where $\EL$ and $m$ are as before, and $d=1$. Then we prove
\begin{theorem}[\bf Global regularity for fractal Burgers]  \label{thm:Burg:global}
Assume that $\theta_0$ is smooth and periodic, and $m$ is such that \eqref{eq:m:critical}--\eqref{eq:m:decay} hold and
 \begin{align}
\lim_{\eps \to 0+} \int_{\eps}^{1} m(r) dr  = \infty. \label{eq:m:Burg}
\end{align}
Then there exists a unique $C^\infty$ smooth solution $\theta$ of the initial value problem associated to \eqref{eq:Burg}.
\end{theorem}
In addition, in the case of the Burgers equation we prove that condition \eqref{eq:m:Burg} is sharp:
\begin{theorem}[\bf Finite time blow-up for fractal Burgers] \label{thm:Burg:blowup}
Assume that $m$ is such that \eqref{eq:m:critical}--\eqref{eq:m:decay} hold, and in addition we have
\begin{align}
r |m'(r)| \leq C m(r) \label{eq:m:assume:Hor}
\end{align}
for $r>0$ and some constant $C\geq 1$. Furthermore, suppose that
 \begin{align}
\lim_{\eps \to 0+} \int_{\eps}^{1} m(r) dr  < \infty \label{eq:m:Burg:BAD}
\end{align}
holds. Then there exists an initial datum $\theta_0 \in C^\infty(\TT)$, and $T>0$ such that $\lim_{t\to T} \|\theta_x(t)\|_{L^\infty} = \infty$, i.e. we have finite time blow up arising from
smooth initial data.
\end{theorem}

A natural class of dissipation terms is associated with Fourier multiplier operators. This representation is closely related to the form \eqref{eq:L:general}. As noted above, when $m(r) =
r^{-\alpha} C_{d,\alpha}$ for a suitable constant $C_{d,\alpha}$, then $\EL = \LL^\alpha,$ corresponding to the Fourier multiplier with symbol $P(\zz) = |\zz|^\alpha$. One may generalize this
statement as follows. Let $P(\zz)$ be a sufficiently nice Fourier multiplier (see Lemmas~\ref{lemma:Fourier:upper} and~\ref{lemma:Fourier:lower} for precise assumptions on $P$), and let $K(y)$
be the convolution kernel associated to the multiplier $P$, i.e. $\hat{\EL \theta}(\zz) = P(\zz) \hat{\theta}(\zz)$, where $\EL$ is the operator defined in \eqref{eq:L:general}. Then there exists
a positive constant $C$ that depends only on $P$, such that \eqref{eq:K:cond} holds for all sufficiently small $y$, with $m(y) = P(1/\zz)$. This turns out to be sufficient to prove an analog of
Theorem~\ref{thm:SQG:global} (and Theorem~\ref{thm:Burg:global}).

\begin{theorem}[{\bf Global regularity for \just supercritical SQG}] \label{thm:SQG:Fourier}
Let $P$ be a radially symmetric Fourier multiplier that is smooth away from zero, non-decreasing, satisfies  $P(0)=0$, $P(\zz) \to \infty$ as $|\zz| \to \infty$, as well as conditions
\eqref{eq:P:cond:doubling}--\eqref{eq:P:cond:Hormander}, and  \eqref{eq:P:cond:lower}. Suppose also
\begin{align}
P(|\zz|) \leq C |\zz| \label{eq:cond:P:regularity:1}
\end{align}
for all $|\zz|$ sufficiently large,
\begin{align}
|\zz|^{-\alpha}P(|\zz|) \mbox{ is non-decreasing} \label{eq:cond:P:regularity:2}
\end{align}
 for some $\alpha>0$, and
\begin{align}
\int_{1}^{\infty} P( |\zz|^{-1} ) d|\zz| = \infty \label{eq:cond:P:regularity:3}.
\end{align}
Then for any $\theta_{0}$ that is smooth and periodic, the Cauchy problem for the dissipative SQG equation \eqref{eq:P:SQG:1}--\eqref{eq:P:SQG:2} has a unique global in time smooth solution.
\end{theorem}

In particular, Theorem~\ref{thm:SQG:Fourier} proves global regularity of solutions for dissipative terms given by multipliers with behavior $P(\zeta) \sim |\zeta|(\log|\zeta|)^{-a}$ for large $\zeta,$ where $0
\leq a \leq 1.$ The details of the assumptions on $P$ and more discussion can be found in Section~\ref{app:Fourier} below.

\section{Pointwise moduli of continuity} \label{sec:prelim:MOC}

\begin{definition}[\bf Modulus of continuity] \label{def:MOC}
We call a function $\omega \colon (0,\infty) \rightarrow (0,\infty)$ a {modulus of continuity} if $\omega$ is non-decreasing, continuous, concave, piecewise $C^2$ with one sided derivatives, and
 additionally satisfies $\omega'(0+)<\infty$ and $\omega''(0+) = -\infty$. We say that a smooth function $f$ {\em obeys} the modulus of continuity $\omega$ if $|f(x) - f(y)| <
\omega(|x-y|)$ for all $x\neq y$.
\end{definition}
We recall that if $f\in C^\infty(\TT^2)$ obeys the modulus $\omega$, then $\Vert \nabla f \Vert_{L^\infty} < \omega'(0)$ \cite{KNV}. In addition, observe that a function $f\in C^\infty(\TT^2)$
automatically obeys any modulus of continuity $\omega(\xi)$ that lies above the function $\min\{\xi \Vert \nabla f \Vert_{L^\infty}, 2 \Vert f \Vert_{L^\infty} \}$.

The following lemma gives the modulus of continuity of the Riesz transform of a given function.
\begin{lemma}[\bf Modulus of continuity under a Riesz transform]\label{lemma:u:MOC}
Assume that $\theta$ obeys the modulus of continuity $\omega$, and that the drift velocity is given by the constitutive law $u =  \nabla^\perp \LL^{-1} \theta$. Then $u$ obeys the modulus of
continuity $\Omega$ defined as
\begin{align}
\Omega(\xi) = A \left( \int_{0}^{\xi} \frac{\omega(\eta)}{\eta} d\eta + \xi \int_{\xi}^{\infty} \frac{\omega(\eta)}{\eta^{2}} d\eta \right) \label{eq:u:moc}
\end{align}
for some positive universal constant $A>0$.

Moreover, for any two points $x ,y$ with $|x - y| = \xi >0$ we have
\begin{align}
 \left| (u(x)-u(y)) \cdot \frac{x-y}{|x-y|} \right| \leq  \tilde\Omega(\xi) +  \Omega^\perp(\xi)    \label{eq:u:dot:moc}
\end{align}
where
\begin{align}
\tilde\Omega(\xi) = A \left( \omega(\xi) + \xi \int_{\xi}^{\infty} \frac{\omega(\eta)}{\eta^2} d\eta \right) \label{eq:Omega:def}
\end{align}
and
\begin{align}
 \Omega^\perp(\xi) = A \int_{0}^{\xi/4} \!\! \int_{\xi/4}^{3\xi/4} \big( \theta(\eta,\nu) - \theta(-\eta,\nu) - \theta(\eta,-\nu) + \theta(-\eta,-\nu) \big) \frac{\nu}{\big( (\xi/2 - \eta)^2 +
 \nu^2 \big)^{3/2}} d\eta d\nu \label{eq:Omega:perp:def}
\end{align}
where $A$ is a universal constant.
\end{lemma}
The proof of \eqref{eq:u:moc} may be found in \cite[Appendix]{KNV}, while \eqref{eq:u:dot:moc} was obtained in~\cite[Lemma 5.2]{K2}, to which we refer for further details.

\begin{lemma}[\bf Dissipation control] \label{lemma:D}
Let $\EL$ be defined as in \eqref{eq:L:general}, with $K$ satisfying \eqref{eq:K:cond}.
Assume that $\theta \in C^\infty(\TT^2)$ obeys a concave modulus of continuity $\omega$. Suppose that there exist two points $x ,y$ with $|x - y| = \xi >0$ such that $ \theta(x) - \theta(y) =
\omega(\xi)$. Then we have
\begin{align}
\EL\theta(x) - \EL\theta(y) \geq  \DD(\xi)  +\DD^\perp(\xi) \label{eq:lemma:D}
\end{align}
where 
\begin{align}
\DD(\xi) &= \frac{1}{A} \int_{0}^{\xi/2} \big( 2 \omega(\xi) - \omega(\xi + 2 \eta) - \omega(\xi-2\eta) \big) \frac{m(2 \eta)}{\eta} d\eta \notag \\
& \qquad + \frac{1}{A} \int_{\xi/2}^\infty \big( 2 \omega(\xi) - \omega(\xi + 2 \eta) + \omega(2\eta - \xi) \big) \frac{m(2 \eta)}{\eta} d\eta \label{eq:D:def}
\end{align}
and
\begin{align}
\DD^\perp(\xi) &= \frac{1}{A}  \int_{0}^{\xi/4}\!\!\! \int_{\xi/4}^{3\xi/4} \big( 2 \omega(2 \eta) - \theta(\eta,\nu) + \theta(-\eta,\nu) - \theta(\eta,-\nu) + \theta(-\eta,-\nu)\big) \notag \\
&\qquad \qquad \qquad \qquad \times \frac{m(\sqrt{(\xi/2-\eta)^2 + \nu^2})}{(\xi/2 - \eta)^2 + \nu^2} d\eta d\nu \label{eq:D:perp:def}
\end{align}
with some sufficiently large universal constant $A>0$ (that we can, for simplicity of presentation, choose to be the same as in Lemma~\ref{lemma:u:MOC}). The corresponding lower bound in one
dimension includes only the $\DD$ term.\end{lemma} Note that $\DD\geq 0$ due to the concavity of $\omega$, while $\DD^\perp \geq 0$ since $\theta$ obeys the modulus of continuity $\omega$. The
above lemma may be obtained along the lines of \cite[Section 5]{K2}, where it was obtained for $\EL=\LL$. However, some modifications are necessary for more general diffusions we consider, 
and we provide a proof 
in Appendix~\ref{app:D} below.
We conclude this section by establishing a bound for $\Omega^{\perp}(\xi)$ in terms of $\DD^{\perp}(\xi)$. 

\begin{lemma}[\bf Connection between $\Omega^{\perp}$ and $\DD^{\perp}$]\label{lemma:perp}
Let $m$ be as in Section~\ref{sec:m:assume}, and assume $\theta$ obeys the modulus of continuity $\omega$. For $\Omega^{\perp}(\xi)$ and $\DD^{\perp}(\xi)$ defined via \eqref{eq:Omega:def} and
\eqref{eq:D:perp:def} respectively we have
\begin{align}
m(\xi) \Omega^{\perp}(\xi) \leq A^2 \DD^{\perp}(\xi) \label{eq:perp:bound}
\end{align}
for all $\xi>0$.
\end{lemma}
\begin{proof}[Proof of Lemma~\ref{lemma:perp}]
To prove \eqref{eq:perp:bound}, first observe that since $\theta$ obeys $\omega$, we have $\theta(\eta,\nu) - \theta(-\eta,\nu) \leq \omega(2\eta)$ and also $\theta(\eta,-\nu)  -
\theta(-\eta,-\nu) \leq \omega(2\eta)$. Therefore we have that
\begin{align}
& \big| \theta(\eta,\nu) - \theta(-\eta,\nu) - \theta(\eta,-\nu) + \theta(-\eta,-\nu) \big| \notag\\ & \qquad \qquad \leq 2 \omega(2 \eta) - \theta(\eta,\nu) + \theta(-\eta,\nu) -
\theta(\eta,-\nu) + \theta(-\eta,-\nu). \label{eq:quad:symmetry}
\end{align}
holds, for any $(\eta,\nu) \in \RR^{2}$.

Next, we claim that for any $0< \nu \leq \xi/4$ and any $|\eta - \xi/2| \leq \xi/4$ we have
\begin{align}
\frac{\nu\; m(\xi)}{\big( (\xi/2 - \eta)^2 + \nu^2 \big)^{3/2}} \leq \frac{m(\sqrt{(\xi/2-\eta)^2 + \nu^2})}{(\xi/2 - \eta)^2 + \nu^2}. \label{eq:m:perp:bound}
\end{align}
To prove \eqref{eq:m:perp:bound}, we observe that in this range for $(\eta,\nu)$ we have $\sqrt{(\xi/2-\eta)^2 + \nu^2} \leq \xi$, and due to the monotonicity of $m$ it follows that $m(\xi) \leq
m( \sqrt{(\xi/2-\eta)^2 + \nu^2} )$. Since $\nu \leq \sqrt{(\xi/2-\eta)^2 + \nu^2}$, \eqref{eq:m:perp:bound} holds. Recalling the definitions of $\Omega^{\perp}$ and $\DD^{\perp}$, it is clear
that \eqref{eq:perp:bound} follows directly from \eqref{eq:quad:symmetry} and \eqref{eq:m:perp:bound}, concluding the proof of the lemma.
\end{proof}

\section{Global regularity for \just supercritical SQG} \label{sec:SQG:global}

In this section we prove Theorem~\ref{thm:SQG:global}. The local well-posedness for smooth solutions to SQG-type equations is by now standard.  In particular, we have:
\begin{proposition}[\bf Local existence of a smooth solution] \label{prop:LWP}
Given a periodic $\theta_0 \in C^\infty$, there exists $T>0$ and a periodic solution $\theta(\cdot,t) \in C^\infty$ of \eqref{eq:SQG:1}--\eqref{eq:SQG:2}. Moreover, the smooth solution may be
continued beyond $T$ as long as $ \Vert \nabla
\theta\Vert_{L^1(0,T;L^\infty)} < \infty$.
\end{proposition}
The local in time propagation of $C^\infty$ regularity may in fact even be obtained in the absence of dissipation, since $u$ is divergence free. Since in \eqref{eq:SQG:1}--\eqref{eq:SQG:2} we
have a dissipative term, one may actually show local $C^\infty$ regularization of sufficiently regular initial data. 
The proof may be obtained in analogy to the usual super-critical SQG~\cite{Dong}, since in view of \eqref{eq:m:decay} $\EL$ is smoothing more than $\LL^\alpha$ for some $\alpha>0.$ 
The presence of the general dissipation $\EL$ instead of the usual $\LL^{\alpha}$ does not introduce substantial difficulties.

The main difficulty in proving Theorem~\ref{thm:SQG:global} is the super-criticality of the dissipation in \eqref{eq:SQG:1}--\eqref{eq:SQG:2}. Thus, as opposed to the critical case \cite{KNV},
here we {\em cannot} construct a single modulus of continuity $\omega(\xi)$ preserved by the equation, and then use the scaling $\omega_{B}(\xi) = \omega(B\xi)$ to obtain a family of moduli of
continuity such that each initial data obeys a modulus in this family. Instead, we will separately construct a modulus of continuity $\omega_{B}(\xi)$ for each initial data, and each such
modulus will be preserved by the equation for all times (see also~\cite{DKV}).

Before constructing the aforementioned family of moduli, let us recall the breakthrough scenario of~\cite{KNV}.
\begin{lemma}[\bf Breakthrough scenario]\label{lemma:break}
Assume $\omega$ is a modulus of continuity such that $\omega(0+)=0$ and $\omega''(0+)= -\infty.$ Suppose that the initial data $\theta_0$ obeys $\omega.$ If the solution $\theta(x,t)$ violates
$\omega$ at some positive time, then there must exist $t_1>0$ and $x \neq y \in \TT^2$ such that
\[ \theta(x,t_1) - \theta(y,t_1) = \omega(|x-y|), \]
and $\theta(x,t)$ obeys $\omega$ for every $0 \leq t<t_1.$
\end{lemma}
Let us consider the breakthrough scenario for a modulus of continuity $\omega.$ A simple computation \cite{K2}, combined with Lemma~\ref{lemma:u:MOC} and Lemma~\ref{lemma:D} yields
\begin{align}
\partial_{t} \left( \theta(x,t) - \theta(y,t) \right)|_{t=t_1} &= u \cdot \nabla \theta (y,t_1) - u \cdot \nabla \theta (x,t_1) + \EL\theta(y,t_1) - \EL\theta(x,t_1)\notag\\
&\leq \left| \big( u(x,t_1)-u(y,t_1) \big) \cdot \frac{x-y}{|x-y|} \right| \omega'(\xi) +  \EL\theta(y,t_1) - \EL\theta(x,t_1)\notag\\ &\leq \min \left\{\Omega(\xi) , \tilde\Omega(\xi)  +
\Omega^\perp(\xi) \right\} \omega'(\xi) -   \big( \DD(\xi) + \DD^\perp(\xi) \big) \label{eq:MOC}
\end{align}
where $\Omega, \tilde\Omega, \Omega^\perp, \DD$, and $\DD^\perp$ are given in \eqref{eq:u:moc}, \eqref{eq:Omega:def}, \eqref{eq:Omega:perp:def},  \eqref{eq:D:def}, and \eqref{eq:D:perp:def}
respectively. If we can show that the expression on the right side of \eqref{eq:MOC} must be strictly negative, we obtain a contradiction: $\omega$ cannot be broken, and hence it is preserved by
the evolution \eqref{eq:SQG:1}.

\subsection{Construction of the family of moduli of continuity} \label{sec:moc:construct}
We now construct a family of moduli of continuity $\omb$, such that given any periodic $C^\infty$ function $\theta_{0}$, there exists $B\geq 1$ such that $\theta_{0}$ obeys $\omb$.

Fix a sufficiently small positive constant $\kappa>0$, to be chosen precisely later in terms of the constant $A$ of \eqref{eq:Omega:def} and the function $m$. For any $B\geq 1$ we define $\delta(B)$ to be the unique solution of
\begin{align}
m(\delta(B)) = \frac{B}{\kappa}. \label{eq:deltaB:def}
\end{align}
Since $m$ is continuous, non-increasing, $m(r) \to + \infty$ as $r \to 0+$, and \eqref{eq:m:decay} holds, such a solution $\delta(B)$ exists for any $B \geq 1$ (if $\kappa$ is small enough). 
For convenience we can ensure that $\delta(B) \leq r_0/4$ for any $B \geq 1$, by using \eqref{eq:m:critical} and choosing $\kappa < r_0/(4 C_{0})$. Note that $\delta(B)$ is a non-increasing function of $B$.

We let $\omb(\xi)$ to be the continuous function with $\omb(0)=0$ and
\begin{align}
&\omb'(\xi)  = B - \frac{B^2}{2 C_\alpha \kappa} \int_{0}^{\xi} \frac{3 + \ln (\delta(B)/\eta)}{\eta m(\eta)} d\eta,  &\mbox{for } 0<\xi< \delta(B), \label{eq:MOC:def:low}\\ &\omb'(\xi) = \gamma m(2 \xi),  &\mbox{for } \xi
> \delta(B), \label{eq:MOC:def:high}
\end{align}
where $C_\alpha = (1+3 \alpha)/\alpha^2$ and $\gamma >0$ is a constant to be chosen later in terms of $\kappa, A$, and the function $m$ (through $C_0, \alpha, r_0$ of assumptions \eqref{eq:m:critical}--\eqref{eq:m:decay}). We
emphasize that neither $\kappa$ nor $\gamma$ will depend on $B$.

Let us now verify that the above defined function $\omb$ is indeed a modulus of continuity in the sense of Definition~\ref{def:MOC}. First notice that by construction $\omb'(0+) = B$ and
$\omb(\xi) \leq B \xi$ for all $0<\xi \leq \delta(B)$.
To verify that $\omb$ is non-decreasing, since $m$ is non-negative, we only need to check that $\omb'>0$ for $\xi < \delta(B)$. This is equivalent to verifying that $\omb'(\delta(B)-) > 0$. Using \eqref{eq:m:decay} and the change of variables $\delta(B)/\eta \mapsto \xi$, we may estimate
\begin{align*}
 \int_{0}^{\delta(B)} \frac{3 + \ln (\delta(B)/\eta)}{\eta m(\eta)} d\eta
\leq \int_{0}^{\delta(B)} \frac{3 + \ln (\delta(B)/\eta)}{\eta^{1-\alpha} \delta(B)^{\alpha}m(\delta(B))} d\eta = \frac{1}{m(\delta(B))} \int_{1}^{\infty} \frac{3 + \ln \xi}{\xi^{1+\alpha}} d\xi = \frac{C_\alpha}{m(\delta(B))},
\end{align*}
where $C_\alpha = (1+3\alpha)/\alpha^2$ may be computed explicitly.
The above estimate and \eqref{eq:deltaB:def}--\eqref{eq:MOC:def:low} imply that
\begin{align}
\omb'(\delta(B)-) \geq B - \frac{C_\alpha B^2}{2 C_\alpha \kappa m(\delta(B))} = \frac{B}{2} \label{eq:omb:deltaB},
\end{align}
which concludes the proof that $\omb'>0$.

In order to verify that $\omb''(0+) = -\infty$, we use \eqref{eq:m:critical} and \eqref{eq:MOC:def:low} to obtain
\begin{align*}
 \omb''(\xi) = - \frac{B^2}{2 C_\alpha \kappa \xi m(\xi)} \left( 3 + \ln \frac{\delta(B)}{\xi} \right)  \leq - \frac{B^2}{2 C_\alpha \kappa C_0} \left( 3 + \ln \frac{\delta(B)}{\xi} \right)
\end{align*}
which is strictly negative for $0< \xi <\delta(B)$, and also converges to $-\infty$ as $\xi \to 0+$.

Since $m$ is non-increasing, the concavity may only fail at $\xi = \delta(B)$. By \eqref{eq:deltaB:def} and \eqref{eq:omb:deltaB} we have
\begin{align*}
\omb'(\delta(B)+) = \gamma m(2 \delta(B))  \leq  \gamma m(\delta(B)) = \frac{\gamma B}{\kappa} \leq \frac{B}{2} \leq \omb'(\delta(B)-)
\end{align*}
provided that $2 \gamma \leq \kappa$, and therefore $\omb$ is concave on $(0,\infty)$. In will also be useful to observe that due to the concavity of $\omb$ and the mean value theorem we
have
\begin{align}
\omb(\delta(B)) \geq \delta(B) \omb'(\delta(B) - ) \geq \frac{\delta(B) B}{2}.  \label{eq:omb@delta}
\end{align}

\subsection{Each initial data obeys a modulus}

In order to show that given any $\theta_0 \in C^\infty(\TT^2)$ there exits $B \geq 1$ such that $\theta_0$ obeys $\omb(\xi)$, it is enough to find a $B$ such that $\omb(\xi) \geq \min\{ \xi \|
\nabla \theta_0\|_{L^\infty}, 2 \| \theta_0\|_{L^\infty} \}$ for all $\xi>0$. Letting $a = 2 \|\theta_0\|_{L^\infty} / \| \nabla \theta_0 \|_{L^\infty}$, due to the concavity of $\omb$ it is
sufficient to find $B \geq 1$ such that $\omb(a) \geq 2 \|\theta_0\|_{L^\infty}$. First, by choosing $B$ large enough we can ensure that $a > \delta(B)$. Then, we have that
\begin{align}
\omb(a) = \omb(\delta(B)) + \int_{\delta(B)}^{a} \omb'(\eta) d\eta \geq \gamma \int_{\delta(B)}^{a} m(2 \eta)  d\eta \to \infty \mbox{ as } \delta(B) \to 0 \label{eq:IC:attained}
\end{align}
due to the assumption \eqref{eq:m:SQG}. Therefore each initial data obeys a modulus from the family $\{ \omb\}_{B \geq 1}$.

\subsection{The moduli are preserved by the evolution}
It is left to verify that the above constructed family of moduli of continuity satisfy
\begin{align}
\min \left\{ \Omega_{B}(\xi), \tilde \Omega_B(\xi) + \Omega^{\perp}_B(\xi) \right\} \omb'(\xi) - \left( \DD_{B}(\xi) + \DD_B^\perp(\xi) \right)  < 0\label{eq:SQG:TODO}
\end{align}
for any $\xi > 0$ and $B \geq 1$. Here $\Omega_{B}$ et al. are defined just as $\Omega$ et al., but with $\omega$ replaced by $\omb$.

\subsection*{The case $\xi \geq \delta(B)$}
First we observe that by Lemma~\ref{lemma:perp} and the fact that $m$ is non-increasing, we have
\begin{align*}
\omb'(\xi) \Omega_{B}^{\perp}(\xi) = \gamma m(2 \xi) \Omega_B^{\perp}(\xi) \leq \gamma m(\xi) \Omega_B^{\perp}(\xi) \leq \gamma A^2 \DD_B^{\perp}(\xi) \leq \DD_B^{\perp}(\xi)
\end{align*}
for all $\xi \geq \delta(B)$ if we choose $\gamma\leq  1/A^2$. In view of \eqref{eq:SQG:TODO} it is left to prove that
\begin{align*}
 \tilde{\Omega}_{B}(\xi) \omb'(\xi) - \DD_{B}(\xi) < 0
\end{align*}
for all $\xi \geq\delta(B)$.  In order to do this we claim that for all $\xi>\delta(B)$ we have
\begin{align}
\omb(2\xi) \leq c_{\alpha} \omb(\xi), \label{eq:concave:high}
\end{align}
where $c_{\alpha} = 1 +(3/2)^{-\alpha}$, and $\alpha>0$ is as in assumption \eqref{eq:m:decay}. Note that by definition $1 < c_{\alpha} < 2$. We postpone the proof of \eqref{eq:concave:high} to
the end of this subsection. Using Lemma~\ref{lemma:D}, \eqref{eq:concave:high}, the concavity and the monotonicity of $\omb$, we may bound $-\DD_{B}$ as
\begin{align}
-\DD_{B}(\xi) &\leq  \frac{1}{A}  \int_{\xi/2}^\infty \left(\omb(\xi + 2 \eta) - \omb(2\eta - \xi) -   \omb(2\xi) - (2-c_{\alpha}) \omb(\xi) \right) \frac{m(2 \eta)}{\eta} d\eta \notag\\ &\leq -
\frac{2-c_{\alpha}}{A} \omb(\xi) \int_{\xi/2}^{\xi} \frac{m(2 \eta)}{\eta} d\eta \leq - \frac{2-c_{\alpha}}{A} \omb(\xi) m(2 \xi). \label{eq:D:high}
\end{align}
We emphasize that for the upper bound \eqref{eq:D:high} only the contribution from $\eta \in (\xi/2 , \xi)$ was used.

On the other hand, integrating by parts the contribution from $\tilde{\Omega}_{B}$ may be rewritten as
\begin{align*}
\frac{\tilde{\Omega}_{B}(\xi)}{A}
= \omb(\xi)+ \xi \int_{\xi}^{\infty} \frac{\omb(\eta)}{\eta^2} d\eta
= 2 \omb(\xi) + \gamma \xi  \int_{\xi}^{\infty} \frac{m(2 \eta)}{\eta}.
\end{align*} Using \eqref{eq:m:decay}, we may bound
\begin{align*}
 \int_{\xi}^{\infty} \frac{m(2 \eta)}{\eta} d\eta \leq \xi^{\alpha} m(2 \xi) \int_{\xi}^{\infty} \frac{1}{\eta^{1+\alpha}} d\eta \leq \frac{m(2 \xi)}{ \alpha}
\end{align*}
where $\alpha>0$ is given. Hence, we obtain
\begin{align}
\frac{\tilde{\Omega}_{B}(\xi)}{A} &\leq  2 \omb(\xi) + \frac{\gamma \xi m(2 \xi) }{\alpha}  \label{eq:OM:high:1}.
\end{align}
Now, for $\delta(B) \leq \xi \leq 2 \delta(B)$, by
\eqref{eq:m:decay} we have
\begin{align}
\frac{\gamma \xi m(2 \xi)}{\alpha} \leq \frac{\gamma}{\alpha} (2 \delta(B))^{1-\alpha} \delta(B)^{\alpha} m(\delta(B)) \leq \frac{2\gamma}{\alpha} \delta(B) \frac{B}{\kappa} \leq \frac{B \delta(B)}{2} \leq \omb(\delta(B)) \leq \omb(\xi)  \label{eq:xi:mxi}
\end{align}
by \eqref{eq:omb@delta}, if $\gamma$ is small.  
On the other hand, for $\xi> 2 \delta(B)$ we have $\xi - \delta(B) \geq \xi/2$ and therefore
\begin{align*}
\omb(\xi) \geq \gamma \int_{\delta(B)}^{\xi}  m(2\eta) d\eta \geq \gamma m(2\xi) (\xi - \delta(B)) \geq \frac{\gamma \xi m(2\xi)}{2} .
\end{align*}
Combining the above estimates with \eqref{eq:OM:high:1} leads to a bound 
\begin{align}
\tilde{\Omega}_{B}(\xi)  &\leq A \left(2 + \frac{2}{\alpha} \right)  \omb(\xi)  \label{eq:OM:high}.
\end{align}
 From \eqref{eq:MOC:def:high},  and the bounds \eqref{eq:D:high} and \eqref{eq:OM:high} we hence obtain
\begin{align*}
\tilde{\Omega}_{B}(\xi) \omb'(\xi) - \DD_{B}(\xi) \leq \left( A \gamma  \frac{2\alpha + 2 }{\alpha}   - \frac{2-c_{\alpha}}{A}  \right) \omb(\xi) m(2\xi) < 0
\end{align*}
for all $\xi\geq \delta(B)$, if we set $\gamma$ small enough, depending only on $A,C,\alpha,$ and $c_{\alpha}$.

\begin{proof}[Proof of estimate \eqref{eq:concave:high}]
To verify \eqref{eq:concave:high} for $\delta(B) \leq \xi \leq 2 \delta(B)$ is straightforward since by the mean value theorem and \eqref{eq:m:decay}, similarly to \eqref{eq:xi:mxi}, we obtain
\begin{align*}
\omb(2\xi) \leq \omb(\xi) + \xi \omb'(\xi)  &= \omb(\xi) + \gamma \xi m(2 \xi)\notag\\
& \leq \omb(\xi) +  \frac{2 \gamma}{\kappa} B\delta(B) \leq \omb(\xi) + (c_{\alpha} - 1) \omb(\delta(B)),
\end{align*}
by choosing $\gamma$ small enough. 

Now for $\xi > 2 \delta(B)$, by \eqref{eq:MOC:def:high} and \eqref{eq:omb@delta} we have
\begin{align*}
c_{\alpha} \omb(\xi) - \omb(2\xi) &=  (c_{\alpha}-1) \omb(\delta(B)) + (c_{\alpha}-1) \gamma \int_{\delta(B)}^{\xi} m(2 \eta) d\eta - \gamma \int_{\xi}^{2\xi} m(2\eta) d\eta \\ &\geq
(c_{\alpha}-1) \frac{B \delta(B)}{2} - \gamma \int_{2\xi - \delta(B)}^{2\xi} m(2 \eta) d\eta \notag\\
& \qquad \qquad + \gamma \int_{\delta(B)}^{\xi} \Big(  (c_{\alpha}-1)m(2 \eta) - m(2 \eta + 2 \xi - 2 \delta(B)) \Big)
d\eta.
\end{align*}
We next note that for $\xi \geq 2 \delta(B)$, due to the monotonicity of $m$, we have
\begin{align*}
 \gamma \int_{2\xi - \delta(B)}^{2\xi} m(2 \eta) d\eta \leq \gamma \delta(B) m(\delta(B)) = \frac \gamma \kappa B \delta(B) \leq(c_{\alpha}-1) \frac{B \delta(B)}{2}
\end{align*}
by letting $\gamma$ be small enough.
We next verify that
\begin{align*}
(c_{\alpha}-1)m(2\eta) \geq  m(2\eta + 2\xi - 2\delta(B))
\end{align*}
holds for all $\eta \in (\delta(B),\xi)$. Using \eqref{eq:m:decay}, and recalling that $c_\alpha = 1+ (3/2)^{-\alpha}$, the above inequality follows once we check that
\begin{align*}
(3/2)^{-\alpha} (2\eta +2 \xi - 2\delta(B))^{\alpha} \geq  (2\eta)^{\alpha}
\end{align*}
holds for all $\eta \in (\delta(B),\xi)$.  But since $\xi > 2 \delta(B)$, we have
\begin{align*}
 \frac{\eta + \xi - \delta(B)}{\eta} \geq 1 + \frac{\xi - \delta(B)}{\xi} \geq \frac 32
\end{align*}
thereby concluding the proof.
\end{proof}

\subsection*{The case $0< \xi \leq \delta(B)$}
For small values of $\xi$ we prefer 
to bound the contribution from the advective term using $\Omega_{B}$ instead of $\tilde \Omega_{B} +
\Omega^{\perp}_{B}$. 
It is sufficient to prove that
\begin{align*}
\Omega_{B}(\xi) \omb'(\xi) - \DD_{B}(\xi) < 0.
\end{align*}
Using the concavity of $\omb$ and the mean value theorem, we may estimate
\begin{align*}
-\DD_{B}(\xi) & \leq \frac 1 A \int_{0}^{\xi/2} \big(  \omb(\xi + 2 \eta) + \omb(\xi-2\eta) - 2 \omb(\xi) \big) \frac{m(2 \eta)}{\eta} d\eta \leq \frac{C}{A} \omb''(\xi) \int_{0}^{\xi/2} \eta m(2\eta) d\eta.
\end{align*}
From \eqref{eq:m:decay} we obtain $\eta^\alpha m(2\eta) \geq (\xi/2)^\alpha m(\xi)$ for $\eta \in (0,\xi/2)$. Since $\omb''(\xi) < 0$ we may further bound
\begin{align}
-\DD_B(\xi) \leq \frac{C}{A} \omb''(\xi) \xi^\alpha m(\xi) \int_{0}^{\xi/2} \eta^{1-\alpha} d\eta \leq \frac{C}{A} \omb''(\xi) \xi^2 m(\xi)  \label{eq:D:low}.
\end{align}

The contribution from the advecting velocity is bounded as
\begin{align}
\frac{\Omega_{B}(\xi)}{A} &= \int_{0}^{\xi} \frac{\omb(\eta)}{\eta} d\eta + \xi \int_{\xi}^{\delta(B)} \frac{\omb(\eta)}{\eta^2} d\eta
+ \xi \int_{\delta(B)}^{\infty} \frac{\omb(\eta)}{\eta^2} d\eta \notag\\ &\leq B \xi + B \xi \ln \frac{\delta(B)}{\xi} + \xi \left( \frac{\omb(\delta(B))}{\delta(B)} + \gamma
\int_{\delta(B)}^{\infty} \frac{m(2\eta)}{\eta } d\eta \right) \label{eq:OM:low:1}.
\end{align}
Here we used that $\omb(\eta) \leq B \eta$ for $\eta \in (0,\delta(B))$. Using \eqref{eq:m:critical}--\eqref{eq:m:decay} and \eqref{eq:deltaB:def} we bound
\begin{align*}
\int_{\delta(B)}^{\infty} \frac{m(2 \eta)}{\eta} d\eta \leq \frac{m(2 \delta(B))}{\alpha} \leq \frac{B}{\alpha \kappa} \leq \frac{B}{\gamma}
\end{align*}
for $\gamma \leq \alpha \kappa$. 
Therefore, \eqref{eq:OM:low:1} gives
\begin{align}
\Omega_{B}(\xi) \leq  A B \xi \left( 3 + \log \frac{\delta(B)}{\xi} \right) \label{eq:OM:low}.
\end{align}
From \eqref{eq:MOC:def:low} and the bounds \eqref{eq:D:low} and \eqref{eq:OM:low}, we obtain
\begin{align}
\Omega_{B}(\xi) \omb'(\xi) - \DD_{B}(\xi) &\leq A B^2 \xi \left( 3 + \log \frac{\delta(B)}{\xi} \right)   +
\frac{C}{A} \xi^2 m(\xi) \omb''(\xi) \notag\\
&\leq A B^2 \xi \left( 3 + \log \frac{\delta(B)}{\xi} \right)   \left( 1 - \frac{C}{2 A^2 C_\alpha \kappa} \right)
\label{eq:check:low:1}
\end{align}
for any $\xi \in (0,\delta(B))$, if we choose $\kappa$ small enough. Here we used the explicit expression of $\omb''$  for small $\xi$. Note that the choice of $\kappa$ is independent of $\gamma$ and $B$, which is essential in order to
avoid a circular argument. This concludes the proof of Theorem~\ref{thm:SQG:global}.

\section{Global regularity vs. finite time blow-up for \just supercritical Burgers} \label{sec:Burg:global}

In this section we prove Theorem~\ref{thm:Burg:global} (global regularity) and Theorem~\ref{thm:Burg:blowup} (finite time blow-up).

\begin{proof}[Proof of Theorem~\ref{thm:Burg:global}] Due to evident similarities to the SQG proof given in Section~\ref{sec:SQG:global} above, we only sketch the necessary modifications. 
See also \cite{KNS} for more details.

First, we note that a modulus of continuity $\omb$ is preserved by \eqref{eq:Burg} if
\begin{align}
\omb(\xi) \omb'(\xi) - \DD_{B}(\xi) < 0 \label{eq:Burg:TODO}
\end{align}
where $\DD_{B}$ is defined as by \eqref{eq:D:def}, with $\omega$ replaced by $\omb$. We will consider exactly the same family of moduli of continuity $\omb$ as in the SQG case,
defined via \eqref{eq:MOC:def:low}--\eqref{eq:MOC:def:high}.

We need to verify that \eqref{eq:Burg:TODO} holds for any $\xi>0$. In the case $\xi \geq \delta(B)$, by using \eqref{eq:D:high}, we have
\begin{align}
\omb(\xi) \omb'(\xi) - \DD_{B}(\xi) \leq \omb(\xi) \gamma m(2 \xi) - \frac{C}{4} \omb(\xi) m(2 \xi) < 0
\end{align}
if $\gamma \leq C/8$. On the other hand, for $\xi \in (0,\delta(B))$, we use \eqref{eq:D:low} and obtain
\begin{align}
\omb(\xi) \omb'(\xi) - \DD_{B}(\xi) &\leq B \xi  \omb'(\xi) + C \xi^{2} m(\xi) \omb''(\xi)\notag\\
& \leq B^{2} \xi - \frac{C B^2 \xi}{2 C_{\alpha} \kappa} \left( 3 + \ln \frac{\delta(B)}{\xi} \right) \leq B^{2} \xi \left( 1- \frac{3 C}{2 C_{\alpha} \kappa} \right) < 0
\end{align}
if $\kappa \leq 3 C / (2 C_{\alpha})$. This concludes the proof of Theorem~\ref{thm:Burg:global}.
\end{proof}

\begin{proof}[Proof of Theorem~\ref{thm:Burg:blowup}]
The proof will use ideas from \cite{DDL}, which builds an appropriate Lyapunov functional to show blow up.
Throughout this section we assume that that
\eqref{eq:m:Burg} holds, i.e.
\begin{align}
\int_{0}^{1} m(r) dr = A < \infty. \label{eq:m:cond:Burg}
\end{align}

Let $\theta_{0} \in C^{\infty}$ be periodic and odd, with $\theta_{0}(0) = 0$. For simplicity we may take $\theta_{0}$ to be $\TT = [-\pi,\pi]$ periodic, and consider that
$r_{0}=1$ in \eqref{eq:m:critical}. It is clear that the proof carries over for any period length and for any value of $r_{0}>0$, with obvious modifications.
Assume the solution $\theta(x,t)$ of \eqref{eq:Burg} corresponding to this initial data lies in $C(0,T;W^{1,\infty})$ for some $T>0$, and is hence $C^{\infty}$ smooth on $[0,T]$.
The Burgers equation preserves oddness of a smooth solution, so that we have
$\theta(0,t) = 0$ for $t \in [0,T]$, and also $\theta(x,t) = - \theta(-x,t)$ for all $x \in \TT$ and $t \in [0,T]$.

Let $w(x)$ be defined as the {\em odd} function with
\begin{align}
&w(x) = 1-x, \mbox{ for } x \in (0,1) \nonumber \\ &w(x) = 0, \mbox{ for } x\geq 1. \label{wdef11}
\end{align}
Associated to this function $w$ we define the Lyapunov functional
\begin{align}
L(t) = \int_{0}^{\infty} \theta(x,t) w(x) dx  = \int_{0}^{1} \theta(x,t) w(x) dx. \label{eq:Lyapunov}
\end{align}
Then, due to the maximum principle $
 \| \theta(\cdot,t)\|_{L^{\infty}} \leq \| \theta_{0} \|_{L^{\infty}}
$ which holds on $[0,T]$, and the definition of $w(x)$, we have
\begin{align}
L(t) \leq \|\theta_{0}\|_{L^{\infty}} \label{eq:L:apriori}
\end{align}
for all $t \in [0,T]$. We will next show, using our assumption that $\theta \in C(0,T;W^{1,\infty}),$ that if $T$ is sufficiently large the bound \eqref{eq:L:apriori} is violated. This shows
 that our assumption has been wrong, and $\theta$ has finite time blow up in the
$W^{1,\infty}$ norm - concluding the proof of Theorem~\ref{thm:Burg:blowup}.

To proceed, we first need the following lemma.
\begin{lemma}\label{wcon}
If $\EL$ is a diffusive operator defined by \eqref{eq:L:def} with $m$ satisfying \eqref{eq:m:cond:Burg} in additional to our usual assumptions, and $w$ is given
by \eqref{wdef11}, then
\[ \int_{\RR} |\EL w(x)|\,dx < \infty. \]
\end{lemma}
\begin{proof}
It is sufficient to estimate the integral over positive $x$ since $\EL w(x)$ is odd.
\subsection*{The case $x \geq 1$} Here we have $w(x) = 0$ and $w$ is odd, hence
\begin{align}
\EL w(x)  = \int_{\RR} \big( w(x) - w(y) \big) \frac{m(x-y)}{|x-y|} dy &= - \int_{0}^{1} w(y) \frac{m(x-y)}{|x-y|} dy - \int_{-1}^{0} w(y) \frac{m(x-y)}{|x-y|} dy \notag\\
&= - \int_{0}^{1} w(y) \frac{m(x-y)}{|x-y|} dy + \int_{0}^{1} w(y) \frac{m(x+y)}{|x+y|} dy \notag\\ &= - \int_{0}^{1} (1-y) \left( \frac{m(x-y)}{|x-y|} - \frac{m(x+y)}{|x+y|} \right) dy.
\label{eq:L:high:1}
\end{align}
Using the mean value theorem and the monotonicity of $m$ we estimate
\begin{align*}
  \left| \frac{m(x-y)}{|x-y|} - \frac{m(x+y)}{|x+y|} \right| \leq 2 y \sup_{r \in [x-y, x+y]} \frac{ r |m'(r)| + m(r)}{r^{2}} \leq 4 C y \frac{m(x-y)}{|x-y|^{2}}.
\end{align*}
But the above bound is only convenient when $x\geq 2$, and in this range we obtain
\begin{align}
\int_{2}^{\infty} |\EL w(x)| dx  &\leq 4 C \int_{2}^{\infty} \int_{0}^{1} (1-y) y \frac{m(x-y)}{|x-y|^{2}} dy dx  \leq 4 C \int_{2}^{\infty}  \int_{0}^{1} (1-y) y \frac{m(1)}{|x-1|^{2}} dy dx
\leq  C m(1). \label{eq:L:2:infty}
\end{align}
On the other hand, when $x \in [1,2]$, it is convenient to work with \eqref{eq:L:high:1} directly. By the monotonicity of $m$ we have that
\begin{align}
|\EL w(x)| \leq \int_{0}^{1} (1-y) \left( \frac{m(x-y)}{|x-y|} +  \frac{m(x+y)}{|x+y|} \right) dy \leq 2 \int_{0}^{1} m(1-y) dy = 2 A
\end{align}
for any $x \in [1,2)$, by using \eqref{eq:m:cond:Burg}. Therefore, $\int_{1}^{2} |\EL w(x)| dx \leq 2 A$, and by using \eqref{eq:L:2:infty} we obtain that
\begin{align}
\int_{1}^{\infty} |\EL w(x)|dx \leq 2 A + C C_{0} =: C_{1}. \label{eq:L:1:infty}
\end{align}

\subsection*{The case $0<x<1$} Here we have $w(x) = 1-x$ and therefore
\begin{align}
\EL w(x) &= \int_{1-x}^{\infty} (1-x) \frac{m(y)}{y} dy + \int_{-x}^{1-x} y \frac{m(y)}{|y|} dy \notag\\
& \qquad + \int_{-x-1}^{-x} (2 + y) \frac{m(y)}{|y|} dy + \int_{-\infty}^{-x-1} (1-x) \frac{m(y)}{|y|} dy\notag\\ &=: T_{1}(x) + T_{2}(x) + T_{3}(x) + T_{4}(x). \label{eq:L:split}
 \end{align}
Using condition \eqref{eq:m:decay} we may easily bound $T_{1}$ and $T_{4}$. More precisely, using a change of variables $y \to -y$ in $T_{4}$ we may write
\begin{align*}
|T_{1}(x) + T_{4}(x)| &= (1-x) \int_{1-x}^{1+x} \frac{m(y)}{y} dy + 2 (1-x) \int_{1+x}^{\infty} \frac{y^{\alpha} m(y)}{y^{1+\alpha}} dy \notag\\ &\leq (1-x) \int_{1-x}^{1+x} \frac{m(1-x)}{1-x}
dy + 2 (1-x) (1+x)^{\alpha} m(1+x) \int_{1+x}^{\infty} \frac{1}{y^{1+\alpha}} dy \notag\\ &\leq 2 x\; m(1-x)  + \frac{2 (1-x) m(1)}{\alpha}
\end{align*}
and therefore
\begin{align}
\int_{0}^{1} |T_{1}(x) + T_{4}(x)| dx \leq 2 \int_{0}^{1} x m(1-x) dx + \frac{2 C_0}{\alpha} \int_{0}^{1} (1-x) dx \leq  2 A + \frac{C_0}{\alpha}. \label{eq:T14}
\end{align}
To bound $T_{2}$ we recall that $m$ is even and hence $ T_{2}(x) = \int_{x}^{1-x} m(y) dy, $ which in turn implies
\begin{align}
\int_{0}^{1} |T_{2}(x)| dx  \leq \int_{0}^{1/2} \int_{x}^{1-x} m(y) dy dx + \int_{1/2}^{1} \int_{1-x}^{x} m(y) dy dx  \leq \int_{0}^{1} m(y) dy = A \label{eq:T2}.
\end{align}
Lastly, due to the monotonicity of $m$ we have that
\begin{align*}
| T_{3}(x) | &\leq 2 \int_{x}^{1} \frac{m(y)}{y} dy +  2 \int_{1}^{x+1} \frac{m(y)}{y} dy+ \int_{x}^{x+1} m(y) dy\notag\\ & \leq 2  \int_{x}^{1} \frac{m(y)}{y} dy + 2 \int_{1}^{2} \frac{m(1)}{y}
dy + \int_{x}^{1} m(y) dy + \int_{1}^{2} m(y) dy\notag\\ & \leq 2 \int_{x}^{1} \frac{m(y)}{y} dy + 2 m(1) \log 2 +  A + m(1)
\end{align*}
and therefore
\begin{align}
\int_{0}^{1}  |T_{3}(x)| dx &\leq 2 \int_{0}^{1} \int_{x}^{1} \frac{m(y)}{y} dy dx + 3 m(1) + A\notag\\
&\leq 2 \int_{0}^{1} \int_{0}^{y} \frac{m(y)}{y} dx dy + 3 C_0 + A \leq  3 (C_0 + A)\label{eq:T3}.
\end{align}
Summarizing \eqref{eq:T14}, \eqref{eq:T2}, and \eqref{eq:T3} we obtain that
\begin{align}
\int_{0}^{1} |\EL w(x)| dx \leq 6 A + 3 C_0+ \frac{C_0}{\alpha} =: C_{2}. \label{eq:L:0:1}
\end{align}
\end{proof}

Coming back to our Lyapunov functional $L(t)$, using the evolution \eqref{eq:Burg} and integrating by parts, we obtain
\begin{align}
\frac{d}{dt} L(t) = \int_{0}^{\infty} \theta_{t}(x,t) w(x) dx  &= \int_{0}^{\infty} \left(  \partial_{x} \frac{\theta(x,t)^{2}}{2} - \EL\theta(x,t) \right) w(x) dx \notag\\
&=- \frac{1}{2} \int_{0}^{1} \theta(x,t)^{2} w_{x}(x) dx - \int_{0}^{\infty} \theta(x,t) \EL w(x) dx. \label{eq:L:ODE1}
\end{align}
Here we employed the identity $\int_{0}^{\infty} \EL \theta(\cdot,t) w = \int_{0}^{\infty} \theta(\cdot,t) \EL w.$ This equality can be derived by using oddness of
{\em both} $\theta$ and $w$, evenness of $m$, and Lemma~\ref{wcon} ensuring $\EL w(x) \in L^1$ (see \cite[(2.8)]{DDL} for more details). Also, the integration by parts in the first term of \eqref{eq:L:ODE1} is justified since by our assumption $\theta$ vanishes as $C |x|$  when  $x \to 0$, with $C = \sup_{[0,T]} \| \nabla \theta(\cdot,t)\|_{L^{\infty}}$.

Now, since $w_{x} = -1$ for $0<x<1$, and using the Cauchy-Schwartz inequality, we obtain
\begin{align*}
L(t)^{2} = \left( \int_{0}^{1} \theta(x,t) w(x) dx \right)^{2} &\leq \int_{0}^{1} \theta(x,t)^{2} dx \int_{0}^{1} w(x)^{2} dx \notag\\ &= \frac{1}{3} \int_{0}^{1} \theta(x,t)^{2}dx = -
\frac{1}{3} \int_{0}^{1} \theta(x,t)^{2}  w_{x}(x) dx.
\end{align*}
Therefore, by \eqref{eq:L:ODE1} on $[0,T]$ we have
\begin{align}
\frac{d}{dt} L(t) \geq  \frac 32 L(t)^{2} - \int_{0}^{\infty} |\theta(x,t)| |\EL w(x)| dx \geq L(t)^{2} - \| \theta_{0}\|_{L^{\infty}} \int_{0}^{\infty} |\EL w(x)| dx. \label{eq:L:ODE2}
\end{align}

By Lemma~\ref{wcon}, we then have
\begin{align}
\frac{d}{dt} L(t) \geq   L(t)^{2} - ( C_{1} + C_{2} ) \|\theta_{0}\|_{L^{\infty}} \label{eq:L:ODE3}.
\end{align}
But \eqref{eq:L:ODE3} implies that $L(t)$ blows up in finite time provided that
\begin{align*}
 0 < L(0)^{2} - (C_{1} +C_{2}) \|\theta_{0}\|_{L^{\infty}}   = \left(\int_{0}^{1} (1-x) \theta_{0}(x) dx \right)^{2} -  (C_{1} +C_{2}) \| \theta_{0} \|_{L^{\infty}}.
\end{align*}
It is easy to design initial data satisfying this condition, and thus leading to finite time blow up. This completes the proof of Theorem~\ref{thm:Burg:blowup}.
\end{proof}

\section{Global regularity with dissipative Fourier multiplier} \label{app:Fourier}

In this section we establish a connection between the global regularity results obtained for \eqref{eq:SQG:1}--\eqref{eq:SQG:2} when the dissipative non-local operators $\EL$ are replaced by
dissipative Fourier multiplier operators, an approach that has been more standard in fluid dynamics. More precisely, we will replace $\EL \theta(x)$ by
\begin{align*}
 \left( P(\zz) \hat{\theta}(\zz) \right)^{\vee}(x)
\end{align*}
for a nice enough radially symmetric Fourier multiplier symbol $P$, and consider the global regularity for the \just supercritcal SQG equation
\begin{align}
&\partial_{t} \theta  + u \cdot \nabla \theta + (P \hat\theta)^{\vee} = 0 \label{eq:P:SQG:1}\\ &u = \nabla^{\perp} \Lambda^{-1} \theta. \label{eq:P:SQG:2}
\end{align}
The setting can be either $\TT^2$ or $\RR^2$  with decaying initial data. In the latter case, an additional argument is needed for Lemma~\ref{lemma:break} to remain valid
due to lack of compactness; see \cite{DDu}. We will focus on the periodic case. Note that working on $\TT^d$ is equivalent to working on $\RR^d$ with
$\theta(x,t)$ extended periodically. We will henceforth pursue this strategy, thinking of Fourier multiplier $P$ and its corresponding convolution kernel $K$
in $\RR^d.$

Intuitively, the Fourier multiplier corresponds to a nonlocal operator $\EL$ as defined in \eqref{eq:L:def}, with $m(y)$ that is comparable to $P(1/|y|)$. We make this connection more
precise in the following two lemmas.

\begin{lemma}[\bf Dissipative operator associated to Fourier multiplier - Upper bound]\label{lemma:Fourier:upper}
Let $P(\zz)=P(|\zz|)$ be a radially symmetric function which is smooth away from zero, non-negative, non-decreasing, with $P(0)=0$ and $P(\zz)\to \infty$ as $|\zz|\to \infty$. In addition assume that
\begin{itemize}
\item[(i)] $P$ satisfies the doubling condition:
\begin{align}
 P(2 |\zz|) \leq c_D P (|\zz|) \label{eq:P:cond:doubling}
\end{align}
for some doubling constant $c_D \geq 1$.

\item[(ii)] $P$ satisfies the H\"ormander-Mikhlin condition:
\begin{align}
\left| \partial_\zz^k P(\zz) \right| |\zz|^{|k|} \leq c_H P(\zz) \label{eq:P:cond:Hormander}
\end{align}
for some constant $c_H \geq 1$, and for all multi-indices $k \in {\mathbb Z}^d$ with $|k|\leq N$, with
$N$ depending only on the dimension $d$ and on the doubling constant $c_D$.

\item[(iii)] $P$ has sub-quadratic growth at $\infty$, i.e.
\begin{align}
\int_0^1 P(|\zz|^{-1}) |\zz| d|\zz| < \infty. \label{eq:P:cond:quadratic}
\end{align}

\end{itemize}
Then the Fourier multiplier operator with symbol $P(\zz)$ is given as a non-local operator defined as the principal value of
\begin{align}
\left( P(\cdot) \hat{\theta}(\cdot) \right)^{\vee}(x) = \int_{\RR^{d}} \big( \theta(x) - \theta(x+y) \big) K(y) dy \label{eq:P:K:def}
\end{align}
and the radially symmetric kernel $K$ satisfying
\begin{align}
|K(y)|  \leq C |y|^{-d} P(|y|^{-1}) \label{eq:Kernel:upper:bound}
\end{align}
for all $y\neq 0$, for some positive constant $C>0$. Similarly $|\nabla K(y) | \leq C |y|^{-d-1} P(|y|^{-1})$ for $y\neq 0$.
\end{lemma}

\begin{proof}[Proof of Lemma~\ref{lemma:Fourier:upper}]
As in Littlewood-Paley theory, consider a smooth, radially symmetric functions  $\varphi$, supported on $1/2 \leq |\zz| \leq 2$, such that
\begin{align}
 1 =  \sum_{j \in {\mathbb Z}} \varphi(2^{-j} \zz) \label{eq:unity}
\end{align}
holds for $\zz \in \RR^{d} \setminus \{ 0\}$. We write $\varphi_{j}(\zz) = \varphi(2^{-j} \zz)$, and note that $P(\zz) \varphi_{j}(\zz)$ is  smooth, compactly supported with $P(0)
\varphi_{j}(0) =0$. Hence
\begin{align*}
 K_{j}(y) = - \int_{\RR^{d}} P(\zz) \varphi_{j}(\zz) e^{i y \cdot \zz} d\zz
\end{align*}
is family of $L^{1}$ kernels, which are smooth at the origin, radially symmetric, and have zero mean on $\RR^{d}$. Thus we may write (in order to avoid principal value integrals we use
double differences)
\begin{align*}
 \left( P(\cdot) \varphi_{j}(\cdot) \hat{\theta}(\cdot) \right)^{\vee} (x) = \int_{\RR^{d}} K_{j}(y) \left(2\theta(x) - \theta(x-y) - \theta(x+y)  \right) dy.
\end{align*}
For $ y\neq 0$, let $j_{0} =  [ \log_{2} |y|^{-1} ]$ and fix $N >d+\log_2 c_D$ to be an even integer. By
\eqref{eq:P:cond:doubling}--\eqref{eq:P:cond:Hormander} we have
\begin{align*}
\sum_{j} |K_{j}(y)|
&= \sum_{j<j_0} \|K_{j}\|_{L^\infty} + \sum_{j\geq j_{0}} \left| \int_{\RR^d} P(\zz) \varphi_j(\zz) e^{i \zz \cdot y} dy\right| \notag\\ &\leq \sum_{j<j_0} \|\hat{K}_{j}\|_{L^1} + \sum_{j\geq
j_{0}} |y|^{-N} \left| \int_{\RR^d} (-\Delta)^{N/2} \big( P(\zz) \varphi_j(\zz) \big)  e^{i \zz \cdot y} dy\right| \notag\\ &\leq \sum_{j<j_0} \int_{\RR^d}  P(\zz) \varphi (2^{-j} \zz) d\zz + C
|y|^{-N} \sum_{j\geq j_{0}}2^{-j N}  \int_{2^{j-1}\leq |\zz| \leq 2^{j+1} } P(\zz) d\zz    \notag\\ &\leq P(2^{j_0}) \sum_{j<j_0} \int_{\RR^d} \varphi (2^{-j} \zz) d\zz + C |y|^{-N} \sum_{j\geq
j_{0}} 2^{-j (N-d)} P(2^{j+1}) \notag\\ &\leq C P(2^{j_0}) \sum_{j<j_0} 2^{jd} +  C |y|^{-N} 2^{-j_0(N-d)} \sum_{j\geq j_{0}} 2^{- (j-j_0) (N-d)} P(2^{j_0}) c_D^{j-j_0} \notag\\ &\leq C
P(|y|^{-1}) |y|^{-d} + C P(|y|^{-1}) |y|^{-d} \sum_{j\geq j_0} 2^{-(j-j_0) (N-d - \log_2 c_D)}
\end{align*}
which shows that the sum $K(y) = \sum_j K_j(y)$ converges absolutely for all $y\neq 0$, and proves \eqref{eq:Kernel:upper:bound}. The purpose of condition \eqref{eq:P:cond:quadratic} is now
evident. For a smooth function  $\theta$ (say at least of class $C^2$), in order to make sense of the integral
\begin{align*}
\int_{|y|\leq 1} K(y) \left( \theta(x-y) + \theta(x+y) - 2\theta(x) \right) dy,
\end{align*}
in view of \eqref{eq:Kernel:upper:bound} we need to assume that $\int_{|y|\leq 1} P(|y|^{-1}) |y|^{-d+2} dy < +\infty$, which is equivalent to \eqref{eq:P:cond:quadratic}. The bound for $|\nabla
K|(y)$ is analogous and we omit further details.
\end{proof}

\begin{lemma}[\bf Dissipative operator associated to Fourier multiplier - Lower bound]\label{lemma:Fourier:lower}
Let the Fourier multiplier $P$ and its associated kernel $K$ be as in Lemma~\ref{lemma:Fourier:upper}. Assume additionally that
\begin{itemize}
\item[(v)] $P$ satisfies
\begin{align}
(-\Delta)^{(d+2)/2} P(\zz) \geq c_H^{-1} P(\zz) |\zz|^{-d-2} \label{eq:P:cond:lower}
\end{align}
for all $|\zz|$ sufficiently large (say larger than $c_0 >0$), for some constant $c_H \geq 1$.
\end{itemize}
Then the kernel $K$ corresponding to $P$ (see \eqref{eq:P:K:def}) may be bounded from below as
\begin{align}  \label{eq:kernellowerbound}
K(y) \geq C^{-1} |y|^{-d} P(|y|^{-1})
\end{align}
for all sufficiently small $|y|$, for some sufficiently large constant $C>0$.
\end{lemma}
\begin{proof}[Proof of Lemma~\ref{lemma:Fourier:lower}]
From our assumptions, the symbol $P$ is a $C^{d+2}$ smooth function except perhaps at the origin. Without loss of generality we can assume $P$ to be smooth at the origin as well. Otherwise, we
can write $P = \tilde P + R$ where $\tilde P$ is a $C^{d+2}$ function everywhere with $\tilde P(\zz) = P(\zz)$ for all $|\zz| > c_0$, $\tilde P(0)=0$ and $(-\Delta)^{N/2} \tilde P$ bounded in
$\RR^d$. The remainder $R$ is a bounded compactly supported function with $R(0)=0$. Therefore, the Fourier multiplier operator with symbol $P$ is the sum of the operators with multipliers
$\tilde P$ and $R$. For $\tilde P$ we apply the proof below and obtain a kernel satisfying \eqref{eq:kernellowerbound}, and for the remainder $R$ we have $(R \hat \theta)^\vee = R^\vee \ast
\theta$ and $R^\vee$ is a bounded, mean zero $L^1$ kernel. Thus, adding $R^\vee$ will not destroy the estimate \eqref{eq:kernellowerbound} for small enough $y$.

If $P$ is smooth near $\zz = 0$, we have that $Q(\zz) = (-\Delta)^{(d+2)/2} P(\zz) \in L^1(\RR^d)$. Indeed, $\int_{|\zz|\leq 1} |Q(\zz)| d\zz$ is finite since $P$ is smooth, while by
\eqref{eq:P:cond:Hormander} and \eqref{eq:P:cond:quadratic} we have
\begin{align*}
 \int_{|\zz|\geq 1} |Q(\zz)| d\zz \leq c_H \int_{|\zz| \geq 1} |\zz|^{-(d+2)} P(\zz) d\zz = c_H \int_1^\infty |\zz|^{-3} P(|\zz|) d|\zz| = c_H \int_0^1 r P(r^{-1}) dr < \infty.
\end{align*}
We may hence define the function $M$, the inverse Fourier transform of $-Q$ as
\begin{align*}
M(y) =- \int_{\RR^d} Q(\zz) e^{i \zz \cdot y} d\zz = - \int_{\RR^d} Q(\zz) \cos(y \cdot \zz) d\zz
\end{align*}
where we have used the fact that $Q$ is radially symmetric and real. Moreover, note that $Q$ has zero mean, since in view of
Lemma~\ref{lemma:Fourier:upper} we have the bound $| {Q}^{\vee} (x) | \leq |x|^{d+2} |{P}^{\vee} (x)| \leq C |x|^2 P(|x|^{-1}) \to 0$ as $|x| \to 0$ since $P$ is sub-quadratic at infinity
\eqref{eq:P:cond:quadratic}. Thus we may rewrite $M(y)$ as
\begin{align}
M(y) &=\int_{\RR^d} Q(\zz) (1 - \cos(y \cdot \zz)) d\zz = \int_{\RR^d} Q(\zz)  (1 - \cos( \zz_1 |y|)) d\zz \label{eq:app:M:def}
\end{align}
by using that $Q$ is radially symmetric. 
In order to appeal to \eqref{eq:P:cond:lower} we further split
\begin{align}
M(y)&=  \int_{|\zz|\leq c_{0}} Q(\zz)  (1 -\cos( \zz_1 |y|)) d\zz + \int_{|\zz|>c_0} Q(\zz)  (1 -\cos( \zz_1 |y|) ) d\zz. \label{eq111}
\end{align}
For all $|y| \leq c_{0}^{-1}$, the first integral in \eqref{eq111} can be estimated by from below by $-C_Q |y|^2,$ where $C_{Q} = \int_{|\zz| \leq c_{0}} |Q(\zz)| d\zz.$
Then using \eqref{eq:P:cond:lower}, 
for $|y| \leq c_{0}^{-1}$ we obtain
\begin{align}
M(y) & \geq -C_Q |y|^2 +c_H^{-1} \int_{|\zz|\geq c_0} |\zz|^{-(d+2)} P(\zz) (1 - \cos( \zz_1 |y|)) d\zz \notag\\ &\geq -C_Q |y|^2 + c_H^{-1} |y|^{2} \int_{|z|\geq c_0|y|} |z|^{-(d+2)} P(z
|y|^{-1})  (1 - \cos(z_1)) dz \notag\\ &\geq -C_Q |y|^2 + c_H^{-1} |y|^{2} \int_{2 \geq |z|\geq 1} |z|^{-(d+2)} P(z |y|^{-1})  (1 - \cos(z_1)) dz \notag \\ & \geq -C_Q |y|^2 + c_H^{-1}
2^{-(d+2)} |y|^{2} P(|y|^{-1}) \int_{2 \geq |z| \geq 1} (1 - \cos z_1) dz \notag\\ &\geq -C_Q |y|^2 + 2 C^{-1} |y|^{2} P(|y|^{-1}) \label{eq:M:lower:bound:1}
\end{align}
for some sufficiently large constant $C>0$ that depends only on $c_{H}$ and $d$. The assumption that $P(\zz) \to \infty$ as $|\zz| \to \infty$, combined with \eqref{eq:M:lower:bound:1} shows that
\begin{align*}
 M(y) \geq C^{-1} |y|^{2} P(|y|^{-1})
\end{align*}
holds for all sufficiently small $|y|.$ 

To conclude, we note that since $\hat M =  - Q = - (-\Delta)^{(d+2)/2} P$,
we have that $K(y) = - P^{\vee}(y) = |y|^{-(d+2)} M(y)$ in the sense of tempered
distributions, and hence we obtain that for sufficiently small $|y|$ the bound $K(y) \geq C^{-1} |y|^{-d} P(|y|^{-1})$ holds, concluding the proof of the lemma.
\end{proof}

\begin{remark}[\bf Examples of symbols $P$]
The conditions \eqref{eq:P:cond:doubling}--\eqref{eq:P:cond:quadratic} that were assumed on the symbol $P$  in order to obtain the upper bound for the associated kernel, are fairly common
assumptions in Fourier analysis.  For all symbols for interest to us in this paper, condition \eqref{eq:P:cond:lower} also naturally holds. The dimension relevant to the
SQG equation is $d=2$. 
When $P(\zz) = |\zz| (\log (|\zz|) )^{-a}$ for sufficiently large $|\zz|$ and
$0<a\leq 1$, corresponding to  \eqref{eq:m:log}, one may verify that $ (-\Delta)^{2}P(\zz) |\zz|^{4} / P(\zz) \to 1$ as $|\zz| \to \infty$, so that we may take $c_{H}=2$ in
\eqref{eq:P:cond:lower} if $c_{0}$ is sufficiently large. Thus condition \eqref{eq:P:cond:lower} is not restrictive for the class of symbols we have in mind.
\end{remark}

The proof of Theorem~\ref{thm:SQG:Fourier} combines the estimates in Lemmas~\ref{lemma:Fourier:upper} and~\ref{lemma:Fourier:lower} above, with the argument given in Section~\ref{sec:SQG:global}. One 
complication arises due to the fact that  \eqref{eq:kernellowerbound} only holds for small enough $|y|$. In fact, for the class of multipliers $P$ that we consider,
positivity of the kernel $K$ is not assured. Because of that, $L^\infty$ maximum principle is no longer available. However, there is an easy substitute which is sufficiently
strong for our purpose.

\begin{lemma}\label{linfty11}
Assume that smooth function $\theta(x,t)$ solves \eqref{eq:P:SQG:2}. Suppose that the kernel $K(y)$ corresponding to the multiplier $P$ via \eqref{eq:P:K:def} satisfies
$|K(y)| \leq C|y|^{-d}P(|y|^{-1})$ for all $y$ and $K(y) \geq C^{-1}|y|^{-d}P(|y|^{-1})$ for all $|y| \leq 2\sigma,$ where $\sigma, C$ are positive constants.
Then there exists $M_{*} = M_{*}(P,\theta_0)$ such that $\|\theta(x,t)\|_{L^\infty} \leq M$ for all $t \geq 0.$
\end{lemma}
\begin{proof}
Letting $M(t) = \|\theta(\cdot,t)\|_{L^\infty}$, we
prove that there exists $M_* \geq M(0)$, sufficiently large, such that $M(t) \leq M_*$ for all $t\geq 0$. If not, then for any fixed $M_*$ there exits a $t_*>0$, such that $M(t_*) = M_*$, and
$\partial_t M(t_*) \geq 0$.  For this fixed $t_*$ let $\bar x$ be a point of maximum for $\theta(\cdot,t_*).$ 
We have
\begin{align}
 \EL_1 \theta( \bar x)
 &\geq \int_{\sigma \leq |y| \leq \infty} (\theta(\bar x) - \theta(\bar x + y)) K(y) dy \notag\\
 &\geq cM_*\int_{\sigma \leq |y|\leq 2\sigma} \frac{P(|y|^{-1})}{|y|^d} dy - C\|\theta \|_{L^2(\TT^d)}
  \left( \int_{\sigma \leq |y|} \frac{P(|y|^{-1})^2}{|y|^{2d}} dy \right)^{1/2} \notag\\
 & \geq c M_*  P((2\sigma)^{-1}) - C \|\theta_0\|_{L^2(\TT^d)} P(\sigma^{-1}) \sigma^{-d/2} \label{eq:L1:lower}
\end{align}
We used that $P \geq 0$ implies $\|\theta(\cdot,t)\|_{L^2(\TT^d)} \leq \|\theta_0\|_{L^2(\TT^d)}$ in the above calculation.
The estimate \eqref{eq:L1:lower} proves that $\partial_{t} M(t_*)$ must be negative if $M_*$ is large enough, depending only on $P$
(though $\sigma$ and other constants) and $\theta_0.$ It follows that $M(t)$ will never exceed the larger of this bound
or $\|\theta_0\|_{L^\infty}.$
\end{proof}

\begin{proof}[Proof of Theorem~\ref{thm:SQG:Fourier}] 

The first two lemmas of this section show that for the multiplier $P$ satisfying \eqref{eq:P:cond:doubling}--\eqref{eq:P:cond:quadratic} and
\eqref{eq:P:cond:lower} we have that $( P \hat \theta)^{\vee}(x) = \int (\theta(x) - \theta(x+y)) K(y) dy$, with $K$ being radial and smooth away from zero.
Moreover, $K$ satisfies $|K(y)| \leq C|y|^{-d}P(|y|^{-1}$ for all $y$ and $K(y) \geq c|y|^{-d}P(|y|^{-1}$ for all $|y| \leq 2\sigma,$ where $C,c,\sigma$ are positive constants
depending only on $P.$

Consider a smooth radially decreasing function $\varphi_{0}(y)$ that is identically $1$ on $|y| \leq \sigma $ and vanishes identically on $|y| \geq 2 \sigma$. We decompose
\begin{align*}
 K(y) = K(y) \varphi_{0}(y) + K(y) (1 - \varphi_0(y)) =: K_1(y) + K_2(y),
\end{align*}
so that
\begin{align*}
 ( P \hat \theta)^{\vee}(x) &= \int_{\RR^d} ( \theta(x) - \theta(x+y)) K_1 (y) + \int_{\RR^d}( \theta(x) - \theta(x+y)) K_2 (y) =: \EL_1 \theta(x) + \EL_2 \theta(x).
\end{align*}
The nonlocal operator $\EL_1$ is of type \eqref{eq:L:general}, since by letting
\begin{align}
& m(r) = C^{-1} P(r^{-1}) \varphi_{0}(r) \label{eq:m:Fourier}
\end{align}
we have that
\begin{align*}
K_{1}(y) \geq m(|y|) |y|^{-d}
\end{align*}
for all $y$ and some $C>0.$ It is clear that the above defined $m$ satisfies properties \eqref{eq:m:critical}--\eqref{eq:m:decay} and \eqref{eq:m:SQG} in view of assumptions
\eqref{eq:cond:P:regularity:1}--\eqref{eq:cond:P:regularity:3} imposed on $P$. Therefore for $\EL_1$ we will be able to directly use the estimate in Lemma~\ref{lemma:D}, which relies only
lower bounds for the kernel associated to $\EL_1$.

On the other hand, we observe that $K_2 \in L^1(\RR^d)$ since $K_2(y)  = 0$ for $y \leq \sigma$, and we have
\begin{align*}
|K_2(y)| \leq |K(y)| \leq C P(|y|^{-1}) |y|^{-d} \leq \sigma^{\alpha} P(\sigma^{-1}) |y|^{-d-\alpha}
\end{align*}
for any $|y|\geq \sigma$, by using \eqref{eq:cond:P:regularity:2}. Let us fix the constant $C_2 = \| K_2 \|_{L^1(\RR^d)}$. Then if $\theta(\cdot,t)$ obeys the modulus of continuity $\omega(\xi)$ it
is clear that
\begin{align}
 |\EL_2 \theta (x,t) -  \EL_2 \theta(y,t)| \leq 2 C_2 \min\{ \omega(\xi), M_* \}, \label{eq:EL:2}
\end{align}
where $M_*$ is the $L^\infty$ norm bound from Lemma~\ref{linfty11}, holds for all $x,y \in \RR^d$, where $|x-y|=\xi$.

Now the argument of Section~\ref{sec:SQG:global} goes through with minor changes. We provide an outline of the argument to verify this.
First,, similarly to \eqref{eq:IC:attained} we may prove that for $B$ large enough (now depending on $\sigma$ as well)  we have $\omb(\sigma) \geq 3 M_* \geq 3 \| \theta(\cdot,t)\|_{L^\infty}$, so that the modulus of continuity can only be broken at values of $\xi \in (0,\sigma)$.  Let $\DD_B$ and $\DD_B^\perp$ be the
bounds obtained from the dissipative operator $\EL_1$ via Lemma~\ref{lemma:D}. Note that the only contribution from the integral defining $\DD_B$ that is used in the estimates is for $\eta \in
(0,\xi)$ (see \eqref{eq:D:high} and \eqref{eq:D:low}), and for us $\xi < \sigma$ so all the bounds on the dissipation given in the proof of Theorem~\ref{thm:SQG:global} require no
modification. Therefore, provided $\kappa$ and $\gamma$ are chosen sufficiently small, we have
\begin{align*}
\min \left\{ \Omega_{B}(\xi), \tilde \Omega_B(\xi) + \Omega^{\perp}_B(\xi) \right\} \omb'(\xi) - \left( \frac{1}{2} \DD_{B}(\xi) + \DD_B^\perp(\xi) \right)  < 0
\end{align*}
for any $B \geq 1$ and $\xi \in (0,\sigma)$, exactly as in the proof of Theorem~\ref{thm:SQG:global}. The proof is hence completed once we show that
the contribution of $\EL_2$ is controlled: 
\begin{align}
2 C_2 \min\{ \omb(\xi), 2 \|\theta\|_{L^\infty} \} \leq  \frac12 \DD_B(\xi)  \label{eq:DB:bound}
\end{align}
for any $\xi \in (0,\sigma)$ and any $B \geq 1$. The range $\xi \in (0,\delta(B))$ is clear since here $\omb(\xi) \leq B \xi$ and by \eqref{eq:D:low} we have
\begin{align*}
\DD_B(\xi) \geq - \frac{C}{A} \xi^{2} m(\xi) \omb''(\xi) = \frac{C B^2}{2 \kappa C_{\alpha} A} \xi \left( 3 + \ln \frac{\delta(B)}{\xi} \right) \geq \frac{3 C}{2 \kappa C_{\alpha} A} B \xi \geq 4 C_2 B\xi \geq 4 C_2 \omb(\xi)
\end{align*}
by letting $\kappa$ be small enough (independent of $B \geq 1$).

We next consider the range $\xi \in (\delta(B), \sigma)$. In view of \eqref{eq:D:high}, we have
$\DD_B(\xi) \geq C \omb(\xi) m(2\xi),$ where $C = (2-c_\alpha)/A$.
Since $P(\zz) \to \infty$ as $|\zz| \to \infty$, we have that $C m(2\xi) \geq 4 C_2$, for all $\xi \in (\delta(B),\kappa)$, for some $\kappa>0$. If $\kappa \geq \sigma$ the
proof is completed, but this cannot be guaranteed, so we have to also consider the case $\kappa < \sigma$.  For $\xi \in (\kappa,\sigma)$, we have
\begin{align}
\DD_B(\xi) \geq C \omb(\xi) m(2\xi)
\geq C \omb(\kappa) m(\sigma)
\geq C m(\sigma) \gamma \int_{\delta(B)}^{\kappa} m(2\eta) d\eta \label{eq:DB}.
\end{align}
By making $B$ large enough we can ensure that the right hand side of \eqref{eq:DB} is larger than $2M_{*},$ completing the proof.
\end{proof}

\section{Global well-posedness for a 2D Euler-type equation with more singular velocity} \label{sec:Euler}

In this section we address the issue of global regularity for the {\em inviscid} active scalar equation
\begin{align}
&\partial_{t} \theta - u \cdot \nabla  \theta = 0 \label{eq:Euler:1} \\ &u = \nabla^{\perp} \LL^{-2} P(\LL) \theta \label{eq:Euler:2}
\end{align}
where the multiplier $P(\zz)=P(|\zz|)$ is a radially symmetric function which is smooth, non-decreasing, with $P(0)=0$ and $P(\zz)\to \infty$ as $|\zz|\to \infty$. In addition, we assume that
$P$ satisfies a doubling property,
\begin{align}
 P(2 |\zz|) \leq c_D P (|\zz|) \label{eq:u:doubling}
\end{align}
for some doubling constant $c_D \geq 1$, that
\begin{align}
|\zz|^{-\alpha}P(|\zz|) \mbox{ is non-increasing}, \label{eq:u:rate}
\end{align}
for some $\alpha \in (0,1)$, and a H\"ormander-Mikhlin type condition
\begin{align}
\left| \partial_\zz^k P(\zz) \right| |\zz|^{|k|} \leq c_H P(\zz) \label{eq:u:Hormander}
\end{align}
holds for some constant $c_H \geq 1$, for all multi-indices $k \in {\mathbb Z}^d$, with $|k|\leq N$, where $N$ depends only on the dimension $d$ and on the doubling constant $c_D$. Condition
\eqref{eq:u:rate} is quite natural in view of \eqref{eq:u:growth} below, while conditions \eqref{eq:u:doubling} and \eqref{eq:u:Hormander} are standard in Fourier analysis.
We remark that while finalizing this paper, we have learned of a recent related work \cite{TE} which proves a result very similar to the one proved in this section under slightly less restrictive assumptions on $P.$

Using the technique of Lemma~\ref{lemma:Fourier:upper}, one may show using \eqref{eq:u:doubling} and \eqref{eq:u:Hormander} that the convolution kernel $K$ corresponding to the operator
$\nabla^\perp \LL^{-2} P(\LL)$, i.e. to the Fourier multiplier $i \zz^\perp |\zz|^{-2} P(|\zz|)$, satisfies the following estimates
\begin{align}
|K(x)| \leq C |x|^{-d+1} P(|x|^{-1}), \quad |\nabla K(x)| \leq C |x|^{-d} P(|x|^{-1}), \quad |\nabla \Delta K(x)| \leq C |x|^{-d-2} P(|x|^{-1}) \label{eq:K:est}
\end{align}
for all $x\neq 0$. Moreover we note that $K$ integrates to $0$ around the unit sphere, and hence convolution with $K$ annihilates constants.

The study of Euler equations with more singular velocities,  \eqref{eq:Euler:1}--\eqref{eq:Euler:2}, was recently initiated by Chae, Constantin, and Wu in \cite{CCW}. They prove the global
global regularity for the Loglog-Euler equation; namely, the prove global regularity in the case that arises when $P(\zz) = [\ln ( 1+ \ln( 1+ |\zz|^2))]^{\gamma},$ for $\gamma \in[0,1]$. Their
approach relies on estimates for the Fourier localized gradient of the velocity for a particular class of symbols. Our aim here is to provide a proof of global regularity for a slightly more
general class of symbols $P$, via the modulus of continuity method. 
The main result of this section is the following:

\begin{theorem}[\bf Global regularity for the $P$-Euler equation] \label{thm:P-Euler}
Let $P$ be a smooth radially symmetric function which is smooth, non-decreasing, with $P(0)=0$ and $P(\zz)\to \infty$ as $|\zz|\to \infty$ and satisfies assumptions
\eqref{eq:u:doubling}--\eqref{eq:u:Hormander}. If $\theta_0$ is periodic and smooth, and we assume that
\begin{align}
\int_1^M \frac{dr}{r \ln (2r) P(r)} \to \infty, \qquad \mbox{as } M\to\infty, \label{eq:u:growth}
\end{align}
then the $P-$Euler equation  \eqref{eq:Euler:1}--\eqref{eq:Euler:2} has a global in time smooth solution.
\end{theorem}

\begin{remark}[\textbf{Integral formulation}]
In fact, our proof provides a stronger result if we state the constitutive law relating $u$ and $\theta$ in terms of an integro-differential operator instead of a Fourier multiplier
\[ u(x) = \int_{\RR^d} \theta(x+y) K(y) dy \]
where $K$ is any kernel which satisfies the hypothesis \eqref{eq:K:est} for any function $P$ for which \eqref{eq:u:doubling}, \eqref{eq:u:rate} and \eqref{eq:u:growth} hold, but not necessarily \eqref{eq:u:Hormander}.
\end{remark}

In the previous sections, we constructed autonomous families of moduli of continuity preserved by the dynamics of the respective equations. In the inviscid case, we will construct a single modulus of continuity and then scale it autonomously. The following lemma makes the above observation precise:

\begin{lemma} [\bf Modulus of continuity under pure transport]\label{lemma:flow}
Let $u$ be a Lipschitz vector field and let $\theta$ solve the transport equation
\begin{equation} \label{eq:transport}
\partial_t \theta + u \cdot \nabla \theta = 0
\end{equation}
If $\theta_0 = \theta(\cdot,0)$ has some modulus of continuity $\omega(\xi)$, then $\theta(\cdot,t)$ has the modulus of continuity $\omega(B(t) \xi)$ where $B(t)$ is given by
\begin{align*}
B(t) = \exp \left( \int_0^t \| \nabla u (\cdot,s)\|_{L^\infty}  \right).
\end{align*}
Equivalently, $B(t)$ solves $B(0)=1$ and $\dot B (t) = \| \nabla u (\cdot,t)\|_{L^\infty}  B(t)$.
\end{lemma}

\begin{proof}[Proof of Lemma~\ref{lemma:flow}]
The solution to the transport equation can be obtained by following the flow of the vector field backwards. Indeed, $\theta(x,t) = \theta_0(X(t))$ where $X$ solves the ordinary differential
equation
\begin{align*}
\dot X(s) = u(X(s),t-s), \quad X(0) = x.
\end{align*}
If $X(t)$ and $Y(t)$ are two such trajectories starting and $x$ and $y$ respectively, from Gr\"onwall's inequality
\begin{align*}
|X(t) - Y(t)| \leq \exp \left( \int_0^t  \| \nabla u (\cdot,s)\|_{L^\infty} \right) |x-y|= B(t) |x-y|.
\end{align*}
Therefore
\begin{align*}
|\theta(x,t)-\theta(y,t)| \leq |\theta_0(X(t))-\theta_0(Y(t))| \leq \omega(B(t)|x-y|)
\end{align*}
which concludes the proof of the lemma.
\end{proof}

\begin{proof}[Proof of Theorem~\ref{thm:P-Euler}]
Let us consider an initial data $\theta_0$ whose Lipschitz, $L^\infty$ and $L^2$ norm are bounded by an arbitrary constant $A$. Applying Lemma \ref{lemma:flow} with $\omega(\xi) = A \xi$, we
obtain that $\theta(\cdot,t)$ obeys the modulus of continuity $A B(t) \xi$, i.e. it is Lipschitz continuous with Lipschitz constant
\begin{align}
\| \nabla \theta(\cdot,t)\|_{L^\infty} \leq A \, B(t), \label{eq:theta:Lip:est}
\end{align}
where $B(0)=1$ and $\dot B = \| \nabla u(\cdot,t)\|_{L^\infty}  \, B(t)$.

By the maximum principle, $\|\theta(\cdot,t)\|_{L^\infty} \leq \|\theta_0\| \leq A $ for any time $t$. Moreover, since $u$ is divergence-free, $\|\theta(\cdot,t)\|_{L^2} \leq \|\theta_0\|_{L^2}
\leq A$ for any time $t$.  In order to combine the last two estimates, we have to estimate the Lipschitz norm of $u$ at time $t$. Let $\varphi(y)$ be a radially non-increasing non-negative
function that is constant $1$ on $|y|\leq 1/2$ and vanishes for $|y|\geq 1$. For some $r\in(0,1)$ to be chosen later, we split the integral defining $\nabla u$ into three pieces to estimate
\begin{align*}
|\nabla u(x)| =  \left| \int_{\RR^d} \nabla K(y) \theta(x+y) dy \right| \leq & \int_{\RR^d} |\nabla (\varphi(y/r) K(y))| \,  |\theta(x+y)-\theta(x)|  dy \\
& + \int_{\RR^d} |\nabla((1- \varphi(y/r))
\varphi(y)K(y))| \, |\theta(x+y)| dy \notag\\ & +\left| \int_{\RR^d} \nabla \big((1-\varphi(y))K(y) \big) \theta(x+y) dy \right|
\end{align*}
Using the bounds on $K$ and its derivatives obtained \eqref{eq:K:est}, and the fact that  $\theta$ is Lipschitz with  constant given by \eqref{eq:theta:Lip:est}, we may further bound
\begin{align}
|\nabla u(x)| &\leq C \int_{|y|\leq r} \frac{P(|y|^{-1})}{|y|^d} |\theta(x+y) - \theta(x)| dy +  C \int_{r/2 \leq |y| \leq 1} \frac{P(|y|^{-1})}{|y|^d} |\theta(x+y)| dy \notag\\ & \qquad \qquad
+ \int_{\RR^d} \left| (-\Delta) \nabla \big( (1-\varphi(y))  K(y) \big) \right| \, | (-\Delta)^{-1} \theta(x+y)| dy\\ &\leq C A B(t) \int_{0}^{r} P(\rho^{-1}) d\rho + C \|\theta_0\|_{L^\infty} \int_{r/2}^{1}
\frac{P(\rho^{-1})}{\rho} d\rho  \notag\\ & \qquad \qquad + C \| (-\Delta)^{-1} \theta\|_{L^\infty(\RR^d)} \int_{|y|\geq 1/2} \frac{P(|y|^{-1})}{|y|^{d+2}} dy \notag\\ &\leq C A B(t) r P(r^{-1})
+ C A P(r^{-1}) \ln \frac 2 r + C A P(2)  \label{eq:nabla:u:estimate}.
\end{align}
In the last inequality above we have additionally used two facts: first, that by \eqref{eq:u:rate} we have $\int_{0}^r P(\rho^{-1}) d\rho \leq C r P(r^{-1})$; and second, that since $\theta$ is
periodic and has zero mean on the torus, we can use the Sobolev inequality and estimate $\| (-\Delta)^{-1} \theta\|_{L^\infty(\RR^d)} = \| (-\Delta)^{-1} \theta\|_{L^\infty(\TT^d)} \leq C
\|\theta\|_{L^2(\TT^d)} \leq C A$. By choosing $r = B(t)^{-1}$ in \eqref{eq:nabla:u:estimate}, which is allowed since $B(0)=1$ and $\dot B \geq 0$, we arrive at
\begin{align*}
\|\nabla u(\cdot, t)\|_{L^\infty} &\leq C A \Big( 1 + P(B(t)) (1+\ln 2 B(t)) \Big).
\end{align*}
Finally we rewrite the differential equation for $B(t)$
\begin{align*}
\dot B(t) = \| \nabla u(\cdot,t) \|_{L^\infty} B(t) \leq CA \Big( 1 + P(B(t)) (1+\ln 2 B(t)) \Big) B(t).
\end{align*}
Clearly this ODE has a global in time solution {\em if and only if}
\begin{align*}
\int_1^\infty \frac 1 {r \ln (2r) P (r) }   dr = \infty
\end{align*}
holds,  which finishes the proof.
\end{proof}

\appendix

\section{Estimate on the dissipative operator at points of modulus breakdown} \label{app:D}

Here we prove Lemma~\ref{lemma:D}. Our argument parallels that of \cite{K2}, but is slightly simpler and more general, as we use the integral representation
of the diffusion generator $\EL$ instead of generalized Poisson kernels employed by \cite{K2}.

\begin{proof}[Proof of Lemma~\ref{lemma:D}]
Due to translation invariance and radial symmetry, we may assume without loss of generality that $x=( \tfrac \xi 2,0)$ and $y=(- \tfrac \xi 2,0)$. For a point $(\eta,\nu) \in \RR^2$ we write $
K(\eta,\nu)$ for the dissipation kernel corresponding to $\EL$. 
Then we have
\begin{align}
 \EL\theta(\tfrac \xi 2,0) - \EL\theta(-\tfrac \xi 2,0)
 & =  \int_{\RR} \int_{\RR} \big( \theta(\tfrac \xi 2,0) - \theta(-\tfrac \xi 2,0) - \theta(\tfrac \xi 2 + \eta,\nu) + \theta(-\tfrac \xi 2 + \eta,\nu) \big) K(\eta,\nu) d\eta d\nu.
 \label{eq:L:K:lower}
 \end{align}
Note that since $\theta$ obeys the modulus of continuity $\omega$, one may bound $ \EL\theta(\tfrac \xi 2,0) - \EL\theta(-\tfrac \xi 2,0)$ from below by the expression on the right side of
\eqref{eq:L:K:lower}, but instead of $K(\eta,\nu)$ with a lower bound for it, such as $m( \sqrt{\eta^{2} +\nu^{2}}) (\eta^2 + \nu^{2})^{-1}$ \eqref{eq:K:cond}. We will henceforth
write $K$ as a shortcut for the latter expression, and assume without loss of generality that $K$ is radially non-increasing and non-negative, since so is $m$ in \eqref{eq:K:cond}.
 \begin{align*}
 \EL\theta(\tfrac \xi 2,0) - \EL\theta(-\tfrac \xi 2,0) & =  \int_{\RR} \int_{\RR} \big( \omega(\xi) - \theta( \tfrac \xi 2 + \eta,\nu) + \theta(-\tfrac \xi 2 - \eta,\nu) \big) K(\eta,\nu) d\eta
 d\nu \\
 & =  \int_{\RR} \int_{- \xi /2}^{\infty} \big( \omega(\xi) - \theta( \tfrac \xi 2 + \eta,\nu) + \theta(-\tfrac \xi 2 - \eta,\nu) \big) K(\eta,\nu) d\eta d\nu \\
  & \qquad + \int_{\RR} \int_{- \xi /2}^{\infty} \big( \omega(\xi) - \theta(-\tfrac \xi 2 - \eta,\nu) + \theta(\tfrac \xi 2 + \eta,\nu) \big) K(-\xi - \eta,\nu) d\eta d\nu \\
  & = \int_{\RR} \int_{- \xi /2}^{\infty} \omega(\xi) \big( K(\eta,\nu) + K(-\xi - \eta,\nu) \big) \\
  &  \qquad \qquad - \big( \theta(\tfrac \xi 2 + \eta,\nu) - \theta(-\tfrac \xi 2 - \eta,\nu) \big) \big( K(\eta,\nu)- K(\xi + \eta,\nu) \big) d\eta d\nu \notag\\
  & = \int_{\RR} \int_{- \xi /2}^{\infty} \omega(\xi) \big( K(\eta,\nu) + K(-\xi - \eta,\nu) \big) \\
  &  \qquad \qquad - \omega(\xi+2\eta) \big( K(\eta,\nu)- K(\xi + \eta,\nu) \big) d\eta d\nu \notag\\
  & \quad +\int_{\RR} \int_{- \xi /2}^{\infty} \omega(\xi) \big( K(\eta,\nu) + K(-\xi - \eta,\nu) \big) \\
  &  \qquad \qquad + \big( \omega(\xi+2\eta) - \theta(\tfrac \xi 2 + \eta,\nu) + \theta(-\tfrac \xi 2 - \eta,\nu) \big) \big( K(\eta,\nu)- K(\xi + \eta,\nu) \big) d\eta d\nu \notag\\
  &=: T^{\parallel} + T^{\perp}.
\end{align*}
Note that $K(\eta,\nu) - K(\xi + \eta,\nu) \geq 0$ for $\eta \geq -\tfrac \xi 2$ due to the monotonicity of $K$ (or that of its lower bound). Hence, using that $\theta$ obeys the modulus of
continuity $\omega$, we see that $T^\perp \geq 0$. To obtain a useful lower bound for $T^\perp$, we only retain the singular piece centered about $\eta =0$. Changing variables $\eta +\xi/2
\mapsto \eta$ we have
\begin{align}
 T^\perp &= \int_{\RR}\int_{0}^{\infty} \big( \omega(2\eta) - \theta(\eta,\nu) + \theta(-\eta,\nu) \big) \big( K(\eta- \tfrac{\xi}{2},\nu) - K(\eta + \tfrac{\xi}{2},\nu) \big) d\eta d\nu.
 \label{eq:T:perp:1}
\end{align}
When $|\nu| \leq \xi/4$ and $|\eta - \xi/2| \leq \xi/4$, using that $m$ is radially non-increasing, we have that
\begin{align}
 K(\eta- \tfrac{\xi}{2},\nu) - K(\eta + \tfrac{\xi}{2},\nu) & = \frac{m\left(\sqrt{(\eta - \tfrac{\xi}{2})^2 + \nu^2}\right)}{(\eta - \tfrac{\xi}{2})^2 + \nu^2} - \frac{m\left(\sqrt{(\eta +
 \tfrac{\xi}{2})^2 + \nu^2}\right)}{(\eta + \tfrac{\xi}{2})^2 + \nu^2}\notag\\
 &\geq m\left(\sqrt{(\eta - \tfrac{\xi}{2})^2 + \nu^2}\right) \left( \frac{1}{(\eta - \tfrac{\xi}{2})^2 + \nu^2} - \frac{1}{(\eta + \tfrac{\xi}{2})^2 + \nu^2}  \right)\notag\\
 &\geq m\left(\sqrt{(\eta - \tfrac{\xi}{2})^2 + \nu^2}\right)\frac{1}{2 \left( (\eta - \tfrac{\xi}{2})^2 + \nu^2\right) } = \frac{ K(\eta- \tfrac{\xi}{2},\nu) }{2}. \label{eq:T:perp:2}
\end{align}
Inserting estimate \eqref{eq:T:perp:2} into expression \eqref{eq:T:perp:1}, and recalling that $\theta$ obeys $\omega$, we obtain
\begin{align*}
 T^\perp &\geq \frac{1}{2} \int_{-\xi/4}^{\xi/4} \int_{\xi/4}^{3\xi/4} \big( \omega(2\eta) - \theta(\eta,\nu) + \theta(-\eta,\nu) \big)  K(\eta- \tfrac{\xi}{2},\nu) d\eta d\nu \notag\\
 &= \frac{1}{2} \int_{0}^{\xi/4} \int_{\xi/4}^{3\xi/4} \big( 2 \omega(2\eta) - \theta(\eta,\nu) + \theta(-\eta,\nu) - \theta(\eta,-\nu) + \theta(-\eta,-\nu) \big)  K(\eta- \tfrac{\xi}{2},\nu)
 d\eta d\nu
 = \frac{\DD^\perp}{2}.
\end{align*}

On the other hand, the dissipation contribution from the direction parallel to $x-y$ may be rewritten as
\begin{align*}
T^\parallel & = \int_{\RR} \int_{- \xi /2}^{\infty} \omega(\xi) \big( K(\eta,\nu) + K(-\xi - \eta,\nu) \big) - \omega(\xi + 2\eta) \big( K(\eta,\nu)- K(\xi + \eta,\nu) \big) d\eta d\nu \\
 & = \int_\RR \int_{-\infty}^{-\xi/2} \big( \omega(\xi) + \omega(-\xi - 2\eta) \big) K(\eta,\nu) d\eta d\nu  + \int_\RR \int_{-\xi/2}^\infty\big( \omega(\xi) - \omega(\xi + 2\eta) \big)
 K(\eta,\nu) d\eta d\nu \\
 & = \int_{-\infty}^{-\xi/2} \big( \omega(\xi) + \omega(-\xi - 2\eta) \big) \tilde{K}(\eta) d\eta + \int_{-\xi/2}^\infty\big( \omega(\xi) - \omega(\xi + 2\eta) \big) \tilde{K}(\eta) d\eta \\
 & = \int_0^{\xi/2} \big( 2 \omega(\xi) - \omega(\xi + 2\eta) - \omega(\xi-2\eta) \big) \tilde{K}(\eta) d\eta +  \int_{\xi/2}^\infty \big( 2 \omega(\xi) - \omega(\xi + 2\eta) + \omega(2\eta-\xi)
 \big) \tilde{K}(\eta) d\eta
\end{align*}
where we have denoted
\begin{align*}
\tilde{K}(\eta) = \int_{\RR} K(\eta,\nu) d\nu.
\end{align*}
Since $\omega$ is concave, the proof of the lemma is concluded once we establish the existence of a positive constant $C$ such that
\begin{align*}
\tilde{K}(\eta) \geq \frac{C m(2 \eta)}{\eta}.
\end{align*}
for all $\eta>0$. But the above estimate is immediate since $m$ is non-increasing, and hence
\begin{align*}
 \int_{\RR} K(\eta,\nu) d\nu \geq \int_{-\eta}^{\eta} K(\eta,\nu) d\nu \geq C m(2 \eta) \int_{-\eta}^{\eta} \frac{d\nu}{\eta^2 + \nu^2} \geq \frac{C m(2 \eta)}{\eta},
\end{align*}
thereby concluding the proof of the lemma.
\end{proof}

\begin{remark}[\bf One-dimensional version]
It is clear that the above proof also holds in the one-dimensional case relevant for the Burgers equation. In fact this case is simpler since there is no need to introduce $\tilde{K}$.
\end{remark}

\noindent {\bf Acknowledgement.} \rm
MD was supported in part by research grant of University of Wisconsin-Madison Graduate School. AK acknowledges partial support of the NSF-DMS grant 1104415. LS was partially supported by NSF grants DMS-1001629, DMS-1065979 and the Sloan fellowship. VV was partially supported by an AMS-Simons travel award.

\end{document}